\documentclass[10pt]{article}

\usepackage{geometry}
\RequirePackage[OT1]{fontenc}
\RequirePackage{amsthm,amsmath}
\RequirePackage[numbers]{natbib}
\RequirePackage[colorlinks,citecolor=blue,urlcolor=blue]{hyperref}

\usepackage{bbm}
\usepackage[mathcal]{euscript}
\usepackage{graphicx}
\usepackage{pstricks}
\usepackage{pstricks-add}
\usepackage{pst-plot}
\usepackage{enumerate}
\usepackage{multicol}
\usepackage{subfigure}
\usepackage{mathrsfs}
\usepackage{amsmath,amssymb}
\usepackage{tikz}
\usepackage{stackengine}
\stackMath

\newcommand{\A}{\mathcal{A}}    
\newcommand{\B}{\mathscr{B}}    
\newcommand{\BB}{\mathbb{B}}    
\newcommand{\C}{\mathbf{C}}    
\newcommand{\CK}{{\C}_{\mathcal{K}}} 

\newcommand{\e}{\varepsilon}   
\newcommand{\E}{\mathcal{E}}    
\newcommand{\EE}{\mathbb{E}}     
\newcommand{\F}{\mathcal{F}}    
\newcommand{\G}{\mathcal{G}} 
\newcommand{\GO}{G_{{\hspace{-0.03cm}}_\O}} 	
\newcommand{\gO}{g_{{\hspace{-0.01cm}}_\O}\!} 	
\newcommand{\GbO}{\GO^{{}^{(\beta)}}\!} 			
\newcommand{\GbOn}{G_{{\hspace{-0.03cm}}_{\O_n}}^{{}^{(\beta)}}}  			
\newcommand{\gbO}{\, g^\beta_{{\hspace{-0.01cm}}_\O}}  					
\newcommand{\g}{\mathfrak{g}}
  
\newcommand{\ind}{\mathbbm{1}} 
\newcommand{\II}{\mathbb{I}}   
\newcommand{\K}{\mathscr{K}} 

\newcommand{\NN}{\mathbb{N}}
\renewcommand{\O}{\mathcal{O}}    %

\renewcommand{\P}{\overline{P}}
\newcommand{\PS}{\mathscr{P}}
\newcommand{\PP}{\mathbb{P}}    

\newcommand{\R}{\mathcal{R}}
\newcommand{\RR}{\mathbb{R}}    
\renewcommand{\S}{\mathcal{S}}

\newcommand{\Ss}{\mathscr{S}}
\newcommand{\support}{\mathop{\rm supp}}
\newcommand{\supp}[1]{\mathrm{supp}\rpar{#1}}
\newcommand{\TO}{{T_\O}} 
\newcommand{\U}{\mathscr{U}}
\newcommand{\VS}{\mathscr{V}}

\newcommand{\XX}{X^{{}^{{}_\clubsuit}}\!\!}
\newcommand{\XXt}{X_t^{{}^{{}_\clubsuit}}}

\newcommand{\Z}{\ZZ^{d,n}}
\newcommand{\Zz}{\ZZ^{2,n}}
\newcommand{\ZZ}{\mathbb{Z}}

\providecommand{\rpar}[1]{\left( #1 \right)}               

\numberwithin{equation}{section}
\theoremstyle{plain}
\newtheorem{theorem}{Theorem}[section]
\newtheorem{lemma}[theorem]{Lemma}
\newtheorem{proposition}[theorem]{Proposition}
\newtheorem{corollary}[theorem]{Corollary}
\newtheorem*{remark}{Remark}
\newtheorem*{example}{Example}
\theoremstyle{definition}
\newtheorem{definition}{Definition}[section]

\begin{document}

\title{Powers of Brownian Green Potentials}

\author{Claude Dellacherie\thanks{Laboratoire Rapha\"el Salem,
UMR 6085, Universit\'e de Rouen, Site du Madrillet, 76801 Saint \'Etienne 
du Rouvray Cedex, France. email: Claude.Dellacherie@univ-rouen.fr} ,
Mauricio Duarte\thanks{Departamento de Matematica, Universidad Andres Bello, 
Republica 220, Piso 2, Santiago, Chile. \break email: mauricio.duarte@unab.cl} ,
Servet Mart\'inez\thanks{CMM-DIM;  Universidad de Chile; UMI-CNRS 2807; 
Casilla 170-3 Correo 3 Santiago; Chile. \break email: smartine@dim.uchile.cl}\and \vspace{-0.5cm}
Jaime San Mart\'in\thanks{CMM-DIM;  Universidad de Chile; UMI-CNRS 2807; 
Casilla 170-3 Correo 3 Santiago; Chile. \break email: jsanmart@dim.uchile.cl} ,
Pierre Vandaele\thanks{CMM-DIM;  Universidad de Chile; UMI-CNRS 2807; 
Casilla 170-3 Correo 3 Santiago; Chile. \break email: pierre-rene-gre.vandaele@ac-lyon.fr}}

\maketitle

\begin{abstract}  In this article we study stability properties of $\gO$, the standard Green kernel
for $\O$ an open regular set in $\RR^d$. In $d\ge 3$ we show that $\gbO$ is again a Green kernel of a 
Markov Feller process, for any power $\beta\in [1,d/(d-2))$. In dimension $d=2$, if $\O$ is an open 
Greenian regular set, we show the same result for $\gbO$, for any $\beta\ge 1$ and for the kernel 
$\exp(\alpha \gO)$, when $\alpha \in (0,2\pi)$.
\end{abstract}

\medskip


\section{Introduction and Main Results}

In this paper we study powers, in the sense of Hadamard, of $\GO$ the standard Green potential associated to
Brownian Motion (BM) on a regular open set $\O\subset \RR^d$, killed when exiting $\O$. These operators
have a kernel which are powers of the standard Green kernel in $\O$.

Most of the time, we will assume that $d\ge 3$, and some extensions will be given for $d=2$. So, unless
we say the contrary, $d$ will be greater than or equal to 3. 
In what follows we denote by 
$G=G_d, g=g_d$ the Green potential and kernel for standard BM in $\RR^d$, that is, for $x\neq y \in \RR^d$ 
$$
g(x,y)=g(0,x-y)=C(d)\|x-y\|^{2-d},
$$
where $C(d)=\frac{\Gamma(d/2-1)}{2\pi^{d/2}}$, and for any $F\in \CK$ we have
$$
GF(x)=\int F(y) g(x,y) \, dy.
$$
We recall that $GF(x)=\EE_x(\int_0^\infty F(B_t) \, dt)$, where $(B_t)$ is a d-dimensional BM. 

If $\O$ is a regular (for BM) open set and $\TO=\inf\{t>0: B_t \in \O^c\}$ is the exiting time for $B$, 
then for $F\in \CK(\O)$
$$
\GO F(x)=\EE_x\left(\int_0^{\TO} F(B_t) \, dt\right)=\EE_x\left(\int_0^\infty F(B_{t\wedge \TO)} \, dt\right).
$$
We denote by $\gO(x,y)$ the density of $\GO$ with respect to Lebesgue measure, which 
for $x\neq y \in \O$ is given by
$$
\gO(x,y)=g(x,y)-\EE_x(g(B_{\TO},y)).
$$
For a parameter $\beta \in \RR_+$ we denote by $\gbO$ the $\beta$-power of $\gO$, that is,
$\gbO(x,y)=(\gO(x,y))^\beta$ and the corresponding operator $\GbO$ defined as 
$$
\GbO F(x)=\int F(y) \gbO(x,y) \, dy.
$$
In what follows we denote by $\hat\O$ the one point compactification of $\O$, and 
given a function $f$ defined on $\O$, we extend it to $\hat \O$ 
by putting $f(\partial)=0$, unless we say the contrary. 
Now, we can state our main results.

\begin{theorem} 
\label{the:1}
Let $\O$ be a regular open set in $\RR^d$, with $d\ge 3$, and $\beta\in \big[1,\frac{d}{d-2}\big)$.
Then, the operator $\GbO$ is the Green potential of a unique Feller Semigroup in $\BB_b(\hat\O)$. 
That is, there exists a unique (in law) Feller process
$(X_t)_t$, with c\`adl\`ag paths on $\hat\O=\O\cup\{\partial\}$, such that for 
any $F\in \C_b(\O)\cap L^1(\O,dx)$ and all $x\in \O$
\begin{equation}
\label{eq:00}
\GbO(F)(x)=\EE_x\left(\int_0^\infty F(X_{t})\, dt\right).
\end{equation}
\end{theorem}
\begin{remark}  
In case $\O$ is bounded, of course the condition $F\in \C_b(\O)\cap L^1(\O,dx)$ is just
$F\in \C_b(\O)$. 
We also note that there is a big difference between $\GbO$ and the operator $H$ whose kernel is given
by 
$$
h(x,y)=g^\beta(x,y)-\EE_x(g^\beta(X_{\TO},y)),
$$
where $X$ is the Feller process with Green kernel $g^\beta$. The fact that $H$ is a Green potential is part
of a general result, and it is the Green potential for the killed process: $Y_t=X_t, t< \TO$.
\end{remark}
In dimension $d=2$, we show the following.
\begin{theorem} 
\label{the:2}
Let $\O$ be a regular Greenian open set in $\RR^2$.
\begin{enumerate}
\item   If $\beta\in \big[1,\infty)$, then the operator $\GbO$ is the Green potential of a 
unique Feller Semigroup. 

\item If $\alpha\in (0,2\pi)$, then the operator $G^{\,exp, \alpha}$ defined in $\C_\K(\O)$ as
$$
G^{\,exp, \alpha}(F)(x)=\int_\O F(y) \exp(\alpha \gO(x,y))\ dy
$$
is the Green potential of a unique Feller Semigroup.
\end{enumerate}
\end{theorem}

The general question is when a function of $\gO$ is again the Green potential
of a Markov process. The above theorems show this happens for powers. In dimension
$d\ge 3$, there is a restriction on these powers, due to integrability conditions. In dimension
$d=2$, we even have that some exponentials of $\gO$ are Green potentials. As we will see, these results
are consequence of similar results shown in \cite{dell2009} (see also \cite{dell2011} and \cite{libroDMSM2014}) for 
potential matrices (potentials of transient finite Markov chains). In addition to powers and exponentials, we know that
$F(a)=a+a^2$ and $F(a)=e^a-1$ preserve potential matrices and so this can be
transfered to the setting of this article. The general problem is open even in the case of matrices. The authors of
\cite{libroDMSM2014} conjecture that this is true for any function, 
which is the Laplace transform of a positive measure, that is, 
$F(a)=\int_0^\infty e^{\zeta a} \, d\mu(\zeta)$. In particular, this should be true for any absolutely monotone function.
We mention here that using the results about potential matrices, in \cite{Eisenbaum2014} it was extended this
stability under powers in the context of continuous (bounded) Green kernels.

\medskip

We mention here the results of this article can be extended to potentials of other transient 
Diffusions $X$, or even more general transient Markov processes, taking values on a bounded open set of $\RR^d$. 
One way to achieve such an extension, is to use an approximation of $X$ by finite Markov Chains with enough
control on the respective approximated potentials.

\medskip

The paper is organized as follows. In Section \ref{sec:1} we consider $\O=\RR^d$, for which the
result is well known. In particular, Theorem \ref{the:1} was proved essentially by Frostman in \cite{frostman1935}
and Riesz in \cite{riesz1937,riesz1951}. The operator $G^{{}^{(\beta)}}\!\!=G_{{\RR^d}}^{{}^{(\beta)}}$ 
is proportional to what is called a Riesz potential, which 
corresponds to the Green potential of a standard Brownian Motion subordinated to 
a $\frac{\alpha}{2}$-stable process where $\alpha=d-\beta(d-2)\in (0,2]$ (see for example \cite{hansen2008}), that is,
for all $F\in \CK, x\in \RR^d$
$$
\begin{array}{ll}
G^{{}^{(\beta)}}(F)(x)\hspace{-0.3cm}&=D\, \EE_x\left(\int_0^\infty F(B_{\eta_t}) dt\right)\\
\\
\hspace{-0.3cm}&=D\frac{\Gamma(\frac{d-\alpha}{2})} {\Gamma(\alpha/2) 2^{\alpha/2}\pi^{d/2}} 
\int_{\RR^d} F(y) \|x-y\|^{\alpha-d}\, dy,
\end{array}
$$
where $(B_t)_t$ is a standard $d$-dimensional BM, $(\eta_t)_t$ is a $\frac{\alpha}{2}$-stable subordinator, normalized to 
$\EE(e^{-\lambda \eta_t})=e^{-t \lambda^{\alpha/2}}$, $(B_t)_t $ and  $(\eta_t)_t$ are independent, and 
$$
D=\left(\frac{\Gamma(\frac{d-2}{2})}{2 \pi^{d/2}}\right)^{\frac{d-\alpha}{d-2}}\Big/
\frac{\Gamma(\frac{d-\alpha}{2})}
{\Gamma(\alpha/2) 2^{\alpha/2}\pi^{d/2}}.
$$
Constant $D$ can be removed by a linear time change $t'=D t$. 

\begin{example} $d=3, \alpha=1, \beta=2$. The Green kernel is in this case $g^2_{\RR^3}(x,y)=C^2(3) \|x-y\|^{-2}$, 
for $x,y \in \RR^3$. The process whose potential is $g_3^2$ can be constructed as follows.
Take $W$ a standard one dimensional BM and consider the passage times, for $t\ge 0$ 
$$
\tau_t=\inf\{s\ge 0: W_s>t\}.
$$
The Laplace transform of $\tau_t$ is  $\EE(e^{-\lambda \tau_t})=e^{\sqrt{2}\, t \lambda^{1/2}}$ and
then $\eta_t=\tau_{t/\sqrt{2}}$ is a normalized  $\frac{1}{2}$-stable subordinator.
So, if we take an independent three dimensional BM $B$ and subordinate it
$$
X_t=B_{\tau_{t/(D\sqrt{2})}},
$$
we get a Feller Process, whose Green potential is $G_{\RR^3}^{(2)}$. Its Green kernel is proportional
to $\|x\|^{-2}$, which formally is the standard Green kernel in $\RR^4$ at the point $z=(x,0)$, and so it is proportional
to the density of the amount of time the 4 dimensional BM $(W,B)$ spends around $z$. We are not aware if there is 
a pathwise explanation of this interpretation.
\end{example}

The case of $\,\O=\RR^d$ is simpler
because $G^{{}^{(\beta)}}\!\!$ is a convolution operator and Fourier analysis
can be used to show the result. In Section \ref{sec:1}, we provide the basic estimations we need and we shall
prove directly that $G^{{}^{(\beta)}}\!\!$ satisfies a suitable version of the Complete Maximum Principle 
(CMP, see Definition \ref{def:CMP}) on $\CK$, which is
one of the main ingredients to prove that a positive operator is the potential of a Feller semigroup. The proof
of this CMP is based on a new inequality characterizing 
potential matrices (see Proposition \ref{pro:3} in the Appendix).

\medskip

In Section \ref{sec:2}, we extend the results, of the previous section, to a general regular bounded domain $\O$.
Following the tools developed in Section \ref{sec:1}, we shall prove that $\GbO$ is the potential of a Ray process.
Then, an extra argument is necessary to show that the set of branching points is empty, to conclude that
actually $\GbO$ is the potential of a Feller process. 

\medskip

In Section \ref{sec:3}, we treat the unbounded case, proving the general result in $d=3$.
In Section \ref{sec:4}, we indicate how to prove the case $d=2$. In section \ref{sec:density}, we prove that
these semigroups have a density with respect to Lebesgue measure. In the Appendix we summarize the 
tools we need from the theory of $M$-matrices and their inverses.

\medskip

The main questions in this article, have some relevance in applications. 
When using Markov chains (or more generaly Markov processes)
to model some phenomena, we usually fit the transition probability $P$ (or the infinitesimal generator).
That is, we put a model on $P$, which in general should be a nonnegative matrix, whose row sums are bounded
by one. Then, we impose other restrictions given by the particular problem. What if we cannot measure $P$ directly, 
but we can only measure $U=(\II-P)^{-1}$, the potential of the associated Markov chain? This happens, for example
in electrical networks. Then, one should give a model for $U$. This is complicated, because it is not simple
to describe which nonnegative matrices $U$ are potentials. This is part of what is known in Linear Algebra
as the inverse $M$-matrix problem. So, if we have a large class of functions that
leave invariant the set of potential matrices, we can model the problem by putting a parametric family in such class
of functions.

\medskip

We denote by $\C(\O)$ the set of continuous functions defined on $\O$, $\C_0(\O)$
the subset of continuous functions vanishing at $\infty$, $\CK(\O)$ the subset of continuous functions
with compact support contained in $\O$.
Notice that $\C(\hat \O)$ is naturally identified to $\C_0(\O)\oplus \ind$, where
the decomposition is $F=F-F(\partial)+F(\partial)$ and $F-F(\partial)\in \C_0(\O)$. 
Given a function $f$ defined on $\O$ we extend it to $\hat\O$ by putting $\overline f(\delta)=0$.
We remark that $F\in \C_0(\O)$ iff $\overline F \in \C(\hat\O)$. 
If $\O$ is bounded, then
$F\in \C_0(\O)$ iff $F\in \C(\O)$ and for all $(x_n)_n\subset \O$ such that $x_n\to x\in \partial \O$, then
$F(x_n)\to 0$. If $\O$ is unbounded, we need to prove also that for all $(x_n)_n\subset \O$,
such that $\|x_n\|\to \infty$ then $F(x_n)\to 0$. We denote by $\BB_b(\O)$ the set of bounded 
measurable functions and $\C_b(\O)$ the set of continuous bounded functions. Finally, in 
sums and integrals a restriction of the form $A\cap p(z)$, where $z$ is the variable
of integration and $p$ is a functional proposition, it is understood as usual as $A\cap\{z: p(z) \hbox{ is true}\}$.

\section{Powers of the Green potential in $\RR^d, d\ge 3$}
\label{sec:1}

In this section we consider $\O=\RR^d$.
In what follows we denote by $\g=\g_d$ the Green kernel for the simple random walk in 
$\ZZ^d$. We recall that $G$ is the Green potential associated to a BM in $\RR^d$, whose density
with respect to Lebesgue measure is
$$
g(x,y)=g(0,x-y)=C(d)\|x-y\|^{2-d},\quad \hbox{ for } x\neq y \in \RR^d,
$$
where $C(d)=\frac{\Gamma(d/2-1)}{2\pi^{d/2}}$.
The following proposition summarizes some well known relations between $\g$ and $g$
(see \cite{Lawler2010} Theorem 4.3.1).

\begin{lemma} 
\label{lem:1}
$ $
\begin{itemize}
\item[(i)]  The Green function $\g$ is bounded and moreover $\g(0,x)\le \g(0,0)<\infty$ for all $x\in \ZZ^d$.

\item[(ii)] $\g$ has the following decay, for $x\in \ZZ^d, x\neq 0$
$$
\g(0,x)=dC(d)\|x\|^{2-d}+ \mathcal{O}(\|x\|^{-d})=
d\,g(0,x)+\mathcal{O}(\|x\|^{-d}).
$$
In particular, there exists a constant $c_0=c_0(d)$, such that, for all $x\in \ZZ^d$
$$
\g(0,x)\le c_0\|x\|^{2-d}.
$$
\end{itemize}
\end{lemma}

Consider now $x\in \RR^d$ and the normalized simple random walk starting from $x$
$$
\S_{t,n}^x=x+\sqrt{\frac{d}{n}} \,\sum_{k=1}^{[nt]} \xi_k,
$$
where the random variables $(\xi_k:\, k\ge 1)$ are i.i.d. with common distribution
$$
\PP(\xi={\bf e})=\frac1{2d},
$$
for ${\bf e}\in \{\pm e_1,\cdots, \pm e_d\}$, with $\{e_1,\cdots,e_d\}$ the canonical basis for
$\RR^d$. We also denote by $\Ss_\ell=\sum_{k=1}^{\ell} \xi_k$, with $\Ss_0 \in \ZZ^d$. 

In what follows we will repeatedly use the following notation $\Z=\sqrt{d/n}\, \ZZ^d$. We also 
denote by  $B_\infty(x,r)=\left\{y\in \RR^d:\, \|x-y\|_\infty< r \right\}$, the
open ball in the $\infty$-norm of $\RR^d$, centered at $x$ with radius $r$. The corresponding
ball for the euclidian distance is denoted by $B(x,r)$. Similarly $d_\infty(y,A)$ is the
distance between $y$ and a set $A$, with respect to the infinite norm.

For a set
$A\subset \RR^d$ we denote by 
$$
A^\boxplus=\{y\in \RR^d:\, d_\infty(y,A)\le \sqrt{d}\}.
$$ 

Assume that $F\in\CK$ 
is a nonnegative function with support $\mathscr{K}=\support(F)$.
Then, for $x=0$, we have 
$$
\begin{array}{l}
\EE\left(\int_0^\infty F(\S_{t,n}) dt\right)=\sum\limits_{\ell\in \NN} \EE\left(\int\limits_{\ell/n}^{(\ell+1)/n} 
F\left(\sqrt{d/n}\,\Ss_\ell\right) dt\right)=
\sum\limits_{\ell\in \NN}\frac1n \EE\left(F\left(\sqrt{d/n}\,\Ss_\ell\right)\right)\\
\\
=\sum\limits_{z\in \ZZ^d} F(z\sqrt{d/n}) \frac1n \sum\limits_{\ell\in \NN} \PP(\Ss_\ell=z)=
\sum\limits_{z\in \ZZ^d} F(z\sqrt{d/n}) \frac1n \g(0,z)\\
\\
=\sum\limits_{w\in \Z} F(w) \frac1n \g(0,w\sqrt{n/d})=
\sum\limits_{w\in \Z} F(w) \left(\frac{d}{n}\right)^{d/2} \, \g(0,w\sqrt{n/d})\frac{n^{d/2-1}}{d^{d/2}}\\
\\
=\frac{F(0)\g(0,0)}{n}+\int_{\RR^d} H_n(y) dy,
\end{array}
$$
where $H_n$ is the simple function given by
$$
 H_n(y)=\sum\limits_{w \in \Z, w\neq 0} F(w) \g(0,w\sqrt{n/d})\frac{n^{d/2-1}}{d^{d/2}}\; 
 \ind_{B_\infty\big(w,\frac12\sqrt{d/n}\big)}(y).
$$
Lemma \ref{lem:1} gives the following bound, with $c_1=c_1(d)=c_0(d)/d$
$$
H_n(y)\le c_1 \|F\|_\infty \ind_{\mathscr{K}^\boxplus}(y) \sum\limits_{w \in \Z, w\neq 0} 
\|w\|^{2-d} \; \ind_{B_\infty\big(w,\frac12\sqrt{d/n}\big)}(y),
$$
Now, if $y\in B_\infty\big(w,\frac12\sqrt{d/n}\big)$ we have $\|y\|=\|y-w+w\|\le \|w\|+\frac{\sqrt{d}}{2}\sqrt{d/n}$. 
The fact that $w \in \Z, w\neq 0$
implies that $\|w\|\ge \sqrt{d/n}$ and so $\|y\|\le \left(1+\frac{\sqrt{d}}{2}\right)\|w\|$, which gives
for $c_2=c_2(d)=c_1(d)\left(1+\frac{\sqrt{d}}{2}\right)^{d-2}$
\begin{equation}
\label{eq:1}
H_n(y)\le c_2 \|F\|_\infty  \ind_{\mathscr{K}^\boxplus}(y) \|y\|^{2-d}\in L^1(dy).
\end{equation}
The asymptotic for $\g$ gives also the pointwise convergence, for all $y\neq 0$
$$
H_n(y)\to F(y) g(0,y).
$$
We conclude, that 
{\small
\begin{eqnarray}
&&\EE\left(\int_0^\infty F(\S_{t,n}) dt\right)=
\sum\limits_{w\in \Z} F(w) \left(\frac{d}{n}\right)^{d/2} \, 
\g(0,w\sqrt{n/d})\frac{n^{d/2-1}}{d^{d/2}}\nonumber\\
\label{eq:1.5}&&\\
&&\stackunder{\longrightarrow} {n\to \infty} \, \int F(y) g(0,y) dy=G(F)(0)=\EE_0\left(\int_0^\infty F(B_t) dt\right),
\nonumber
\end{eqnarray}
}a well known fact. These arguments have been included because they serve as a model for the general case
in Theorem \ref{the:1}.

\medskip

We denote by $\G^{n,(\beta)}$ the operator 
$$
\G^{n,(\beta)}(F)(x)=\sum\limits_{w\in \Z} F(w+x) \left(\frac{d}{n}\right)^{d/2} \, 
\left[\g(0,w\sqrt{n/d})\frac{n^{d/2-1}}{d^{d/2}}\right]^\beta,
$$
which is well defined for all $F\in \C_{\mathcal{K}}$ and all $x\in \RR^d$ because 
$\G^{n,(\beta)}(F)(x)$ contains, for every $x$, a finite number of terms.
Notice that if
$x\in \Z$ then
$$
\G^{n,(\beta)}(F)(x)=\sum\limits_{w\in \Z} F(w) \left(\frac{d}{n}\right)^{d/2} \, 
\left[\g\left(x\sqrt{n/d},w\sqrt{n/d}\right)\frac{n^{d/2-1}}{d^{d/2}}\right]^\beta,
$$
Recall that, we have defined $G^{(\beta)}$ 
as the operator acting in $\C_{\mathcal{K}}$ given by
$$
G^{(\beta)}(F) (x)=\int_{\RR^d} F(y+x) [g(0,y)]^\beta dy=\int_{\RR^d} F(y) [g(x,y)]^\beta dy.
$$
In what follows, we denote by $\hbox{osc}=\hbox{osc}_F$ the oscilation of $F$, which is given by 
$$
\hbox{osc}(\delta)=\sup\limits_{x,y:\, \|x-y\|_\infty\le \delta}\{|F(x)-F(y)|\}\le 2\|F\|_\infty
$$ 
for any $\delta>0$. The fact that $F$ is uniformly continuous on $\RR^d$ implies that
$\hbox{osc}(\delta)\to 0$ as $\delta\downarrow 0$. Using the ideas developed before, 
we prove the following important result.

\begin{proposition} 
\label{pro:1}
Assume that $1\le \beta<\frac{d}{d-2}$ and $F\in \C_{\mathcal{K}}$ with $\mathscr{K}=\supp F$. 
Given $x,y\in \RR^d$ such that $\|x-y\|_\infty<\frac12\sqrt{d}$.
Then,
\begin{equation}
\label{eq:5}
|\G^{n,(\beta)}(F)(x)-\G^{n,(\beta)}(F)(y)|\le \hbox{osc}(\|x-y\|_\infty) \Gamma(F),
\end{equation}
where
$\Gamma(F)=\Big(\g^\beta(0,0) d^{\frac{d}{2}(1-\beta)}+ c_3
\sup\limits_{x\in \RR^d}\int\limits_{\mathscr{K}^{2\boxplus}} \|x-y\|^{\beta(2-d)} \, dy\Big)<\infty$, 
$c_3=c_3(d,\beta)=\Big(\frac{c_0(d)}{d} (2+\sqrt{d})^{d-2}\Big)^\beta$ and 
$\mathscr{K}^{2\boxplus}=(\mathscr{K}^{\boxplus})^\boxplus$.

Similarly, it holds for all $y$
\begin{equation}
\label{eq:2}
|\G^{n,(\beta)}(F)(y)|\le 3\Gamma(F) \|F\|_\infty.
\end{equation}

Finally, we have the convergence: If $(z_n)_n$ is any sequence converging to $y$, then
$$
\lim\limits_{n\to \infty} \G^{n,(\beta)}(F)(z_n)=G^{(\beta)}(F) (y)=\int F(w) [g(y,w)]^\beta dw.
$$
\end{proposition}
\begin{proof} Let us prove \eqref{eq:5}. Assume $w\in \Z$. If $w+x\notin \K$ and  $w+y\notin \K$
then $F(w+x)-F(w+y)=0$. So, if this difference is not zero then
$w \in \Z\cap((\K-x) \bigcup (\K-y))$. Consider now, $y(n)\in \Z$ one of the closets elements of $\Z$ to $y$, in 
the infinity norm. In particular, $\|y-y(n)\|_\infty \le \frac12\sqrt{d/n}\le \frac12\sqrt{d}$. It is straightforward
to show that 
$$
(\K-x) \bigcup (\K-y)\subset (\K^\boxplus-y(n)).
$$ 
Denote by $A_n= \Z\cap\K^\boxplus$. Since $|F(w+x)-F(w+y)|\le \hbox{osc}(\|x-y\|_\infty)$, we get
{\small
$$
\begin{array}{l}
|\G^{n,(\beta)}(F)(x)-\G^{n,(\beta)}(F)(y)|\le \hbox{osc}(\|x-y\|_\infty)
\sum\limits_{w\in A_n-y(n)}  \left(\frac{d}{n}\right)^{d/2} \, 
\left[\g\left(0,w\sqrt{n/d}\right)\frac{n^{d/2-1}}{d^{d/2}}\right]^\beta\\
\\
\le \hbox{osc}(\|x-y\|_\infty)\left(\frac{\g^\beta(0,0)}{n^\gamma}d^{\frac{d}{2}(1-\beta)}+
\sum\limits_{w\in A_n \atop w\neq y(n)}  \left(\frac{d}{n}\right)^{d/2} \, 
\left[\g\left(y(n)\sqrt{n/d},w\sqrt{n/d}\right)\frac{n^{d/2-1}}{d^{d/2}}\right]^\beta\right)\\
\\
\le \hbox{osc}(\|x-y\|_\infty)\left(\frac{\g^\beta(0,0)}{n^\gamma}d^{\frac{d}{2}(1-\beta)}+
c_1^\beta \sum\limits_{w\in A_n \atop w\neq y(n)}  \left(\frac{d}{n}\right)^{d/2} \, 
\|y(n)-w\|^{\beta(2-d)}\right),
\end{array}
$$
}
where $\gamma={\frac{d-\beta(d-2)}{2}}>0$.
For $z\in B_\infty\big(w,\frac12\sqrt{d/n}\big)$, we have
$$
\|y(n)-z\|\le \|y(n)-w\|+\|w-z\|\le \frac{\sqrt{d}}{2}\sqrt{d/n}+\|y(n)-w\|.
$$
Since $y(n)\neq w$ and both belong to $\Z$, we conclude that $\|y(n)-w\|\ge \sqrt{d/n}$ and
$$
\|y(n)-z\|\le\left(1+\sqrt{d}/2\right)\|y(n)-w\|.
$$
Thus, we obtain \eqref{eq:5} from
{\small
$$
\begin{array}{l}
|\G^{n,(\beta)}(F)(x)-\G^{n,(\beta)}(F)(y)|\le\\
\\
\le \hbox{osc}(\|x-y\|_\infty)\left(\frac{\g^\beta(0,0)}{n^\gamma}d^{\frac{d}{2}(1-\beta)}+
c_3 \sum\limits_{w\in A_n} \int_{B_\infty\big(w,\frac12\sqrt{d/n}\big)}
\|y(n)-z\|^{\beta(2-d)}\, dz\right)\\
\\
\le \hbox{osc}(\|x-y\|_\infty)\left(\frac{\g^\beta(0,0)}{n^\gamma}d^{\frac{d}{2}(1-\beta)}+
c_3 \int_{\K^{2\boxplus}} \|y(n)-z\|^{\beta(2-d)}\, dz\right)\\
\\ 
\le \hbox{osc}(\|x-y\|_\infty)\left(\g^\beta(0,0)d^{\frac{d}{2}(1-\beta)}+
c_3 \sup\limits_{u\in\RR^d}\int_{\K^{2\boxplus}}
\|u-z\|^{\beta(2-d)}\, dz\right)=\hbox{osc}(\|x-y\|_\infty) \, \Gamma(F),
\end{array}
$$
}where $\K^{2\boxplus}=(\K^\boxplus)^\boxplus$, and 
$c_3=c_3(d,\beta)=c_1^\beta(d)\left(1+\sqrt{d}/2\right)^{\beta(d-2)}$. 
In particular, we have
$$
|\G^{n,(\beta)}(F)(x)-\G^{n,(\beta)}(F)(y)| \le 2\|F\|_\infty\Gamma(F).
$$
We obtain in a similar way 
$$
\begin{array}{ll}
|\G^{n,(\beta)}(F)(y(n))|\hspace{-0.3cm}&\le \|F\|_\infty\left(\g^\beta(0,0)d^{\frac{d}{2}(1-\beta)}+
c_3 \sup\limits_{u\in\RR^d}\int_{\K^{\boxplus}}\|u-y\|^{\beta(2-d)}\, dy\right)\\
\hspace{-0.3cm}&\le  \|F\|_\infty \Gamma(F),
\end{array}
$$
and inequality \eqref{eq:2} is shown.

Following the same ideas as in the proof of \eqref{eq:1.5}, we get for any $y$
$$
\lim\limits_{n\to\infty} \G^{n,(\beta)}(F)(y(n))=G^{(\beta)} F(y).
$$
The result is shown by using \eqref{eq:5}.
\end{proof}

\begin{corollary} 
\label{cor:1}
Under the hypothesis of Proposition \ref{pro:1} we have
$$
\lim\limits_{n\to\infty} \sum\limits_{x\in \Z} \left(\G^{(n,\beta)}(F)(x)-1\right)^+ \, F(x) \left(\frac{d}{n}\right)^{d/2}=
\int_{\RR^d} \left(G^{(\beta)}(F)(x)-1\right)^+ \, F(x) \, dx,
$$
\end{corollary}

\begin{proof} 
Proposition \ref{pro:1} shows that the sequence of functions
$$
\left(\G^{(n,\beta)}(F)(\bullet)-1\right)^+ \, F(\bullet) 
$$
converge pointwise to $\left(G^\beta(F)(\bullet)-1\right)^+ F(\bullet)$, and (\ref{eq:2}) provides the domination
we need to show that 
$$
\lim\limits_{n\to \infty} \int \left(\G^{(n,\beta)}(F)(x)-1\right)^+ \, F(x)\, dx
=\int_{\RR^d} \left(G^{(\beta)}(F)(x)-1\right)^+ \, F(x) \, dx.
$$

On the other hand, if $x\in \Z$ and $y\in B_\infty\big(x,\frac12\sqrt{d/n}\big)$ we have from 
\eqref{eq:5}
$$
|\G^{n,(\beta)}(F)(x)-\G^{n,(\beta)}(F)(y)|\le \hbox{osc}(\|x-y\|_\infty) \Gamma(F),
$$
which together with \eqref{eq:2} gives, for $A_n= \Z\cap \K^\boxplus$
{\small
$$
\begin{array}{l}
\Big|\int \left(\G^{(n,\beta)}(F)(y)-1\right)^+ F(y)\, dy-\hspace{-0.3cm}
\sum\limits_{x\in \ZZ^{d,n}} \left(\G^{(n,\beta)}(F)(x)-1\right)^+ F(x) \left(\frac{d}{n}\right)^{d/2}\Big|\\
\le \sum\limits_{x\in A_n} 
\int\limits \left|\left(\G^{(n,\beta)}(F)(y)-1\right)^+F(y)-\left(\G^{(n,\beta)}(F)(x)-1\right)^+F(x)\right| 
\ind_{B_\infty\big(x,\frac12\sqrt{d/n}\big)}(y)\, dy\\
\\
\le 4 \Gamma(F)\|F\|\,\hbox{osc}\big(\frac12\sqrt{d/n}\big)\, \mu\Big(\mathscr{K}^{2\boxplus}\Big),
\end{array}
$$
}
where $\mu$ is the Lebesgue measure and the result follows.
\end{proof}
\medskip

Our next step it to show $\sum\limits_{x\in \Z} \left(\G^{(n,\beta)}(F)(x)-1\right)^+ 
\, F(x) \left(\frac{d}{n}\right)^{d/2}\ge 0$ for all $n,\beta\ge 1$ and all $F\in \C_{\mathcal{K}}$. 
This will be a consequence of results on potential matrices 
(see \cite{dell2009} or \cite{libroDMSM2014}) and Proposition \ref{pro:3} in Appendix \ref{app:1}.
To use them, we first need to prove the following Lemma.

\begin{lemma}
\label{lem:2}
Assume that $A\subset \Z$ is a finite nonempty set. Then the matrix $U$ defined as, for $x,w\in A$
$$
U_{x,w}=\g(x\sqrt{n/d},w\sqrt{n/d}),
$$
is a nonsingular symmetric potential.
\end{lemma}
As usual we have taken a particular order on $A$ to define $U$, for example the lexicographical order.
\begin{proof} Consider a large integer $N$ such that $E=\sqrt{n/d} \, A\subset [-(N-1),(N-1)]^d \cap \ZZ^d$.
We define the following symmetric transition matrix $P$ indexed by $J=[-N,N]^d \cap \ZZ^d$ 
$$
P_{jk}=\PP_j(R_J<\infty, \Ss_{R_J}=k),
$$
where $R_J=\inf\{p\ge 1: \Ss_p\in J\}$ is the first strict hitting time to $J$, for the process
$\Ss$, and $\PP_j$ indicates that $\Ss_0=j$. If $\hbox{int}(J)$ denotes
the points in $J$ for which all their neighbours in $\ZZ^d$ belong to $J$ (the interior of $J$), then 
$P_{jk}=\frac1{2d}$ for $j\in \hbox{int }(J), k\in J, \|j-k\|_\infty=1$, while $P_{jk}>0$ for
all $j,k \in \partial J=J\setminus \hbox{int}(J)$. We also notice that
$$
\sum_{k\in J} P_{jk}=\begin{cases} 1&\hbox{if } j\in \hbox{int }(J)\\
\PP_j(R_J<\infty)<1 &\hbox{if } j\in \partial J.
\end{cases}
$$
Notice that $P$ is irreducible and strictly substochastic at least at one vertex.
Thus, the matrix $M=\II-P$ is nonsingular (is an M-matrix) and its inverse $V=(\II-P)^{-1}$ is just the
Green potential, restricted to $J$, for the standard random walk
$$
V_{j k}=\EE_j\left(\sum_\ell \ind_{k}(\Ss_\ell)\right)=\g(j,k).
$$
Our matrix $U$ is a principal submatrix of $V$, that is $U=V\big|_{E\times E}$.
Thus $U$ is a nonsingular potential, which corresponds to the potential of the standard random
walk on $\ZZ^d$, restricted to $E$.
\end{proof}

\begin{proposition} 
\label{pro:2}
For all $n\ge 1, \beta\ge 1$ and all $F\in \C_{\mathcal{K}}$ it holds
$$
\sum\limits_{x\in \Z} \left(\G^{(n,\beta)}(F)(x)-1\right)^+ \, F(x) \left(\frac{d}{n}\right)^{d/2}\ge 0.
$$
\end{proposition}

\begin{proof} Consider $F\in \C_{\mathcal{K}}$ and denote by 
$\mathscr{K}$ its support. We consider a large $a>0$ such that 
$\mathscr{K}\subset [-a,a]^d$. Consider $U$
the symmetric matrix indexed by $A=\Z \bigcap [-a,a]^d$, given by
$$
U_{x,w}=c(n,d,\beta)\, \g\left(x\sqrt{n/d},w\sqrt{n/d}\right),
$$
with
$$
c(n,d,\beta)=\left(\frac{d}{n}\right)^{\frac{d}{2\beta}} \frac{n^{d/2-1}}{d^{d/2}}.
$$

This is a nonsingular symmetric potential matrix according to the previous Lemma. 
Then, Proposition \ref{pro:4} shows that its Hadamard power $U^{(\beta)}$ is again
a nonsingular symmetric potential matrix. Finally, Proposition \ref{pro:3} allow us to conclude 
that for all $v\in \RR^A$
$$
0\le \left(\frac{d}{n}\right)^{\frac{d}{2}}\langle (U^{(\beta)} v-1)^+, v\rangle=\left(\frac{d}{n}\right)^{\frac{d}{2}}
\sum\limits_{x\in A}\left(\sum\limits_{w\in A} (U_{x,w})^{\beta}\, v_w-1\right)^+ v_x.
$$
The result follows by taking $v\in \RR^A$ with $v_x=F(x)$.
\end{proof}

Now, we introduce the version of the maximum principle suitable for our purposes.
\begin{definition}
\label{def:CMP}
A positive linear bounded operator $V$ defined on $\CK$ is said to satisfy the Complete Maximum
principle on  $\CK$ if for any  $F\in \CK$ it holds: if $V(F)(x)\le 1$ whenever $F(x)\ge 0$, 
then $V(F)(x)\le 1$ for all $x\in \O$.
\end{definition}

We are in a position to prove the main result of this section.

\begin{theorem} Assume that $1\le \beta < \frac{d}{d-2}$, then the operator $G^{(\beta)}$ satisfies the CMP
on $\CK$.
\end{theorem}

\begin{proof} Consider $F\in \C_{\mathcal{K}}$. It is clear that $G^{(\beta)}(F)$ is a continuous
function vanishing at $\infty$.
Assume that $G^{(\beta)}(F)(x)\le 1$ for those $x\in \RR^d$ such that $F(x)\ge 0$.
From Corollary \ref{cor:1} and Proposition \ref{pro:2}, we have
$$
0\le \int\limits_{\RR^d} \left(G^{(\beta)} (F) (x)-1\right)^+ F(x) \, dx=\int\limits_{x: \, F(x)<0} 
\left(G^{(\beta)} (F) (x)-1\right)^+ F(x) \, dx.
$$
We conclude that $G^{(\beta)}(F)(x)\le 1$ hold a.s. The continuity of $G^{(\beta)}(F)$ shows that  
$G^{(\beta)}(F)\le 1$ and the result follows.

\end{proof}

\section{Powers of the Green potential in bounded regular open sets}
\label{sec:2}

In this section we consider $\O\subset \RR^d$ a bounded regular open set. Let us introduce some
of the basic notation we need. The Green kernel for Brownian Motion in $\O$ is
$$
\gO(x,y)=g(x,y)-\EE_x(g(B_\TO,y)),
$$
for $x,y \in \O$, where $\TO=\TO(B)$ is the exiting time of $\O$ for the Brownian Motion $B$. 
The associated Green operator is denoted by $\GO$.

In the same spirit, if $E\subset \ZZ^d$ we denote by $\g_E$ the Green kernel associated to the random walk, 
killed upon leaving $E$, which is given by
$$
\g_E(x,y)=\EE_x\left(\sum\limits_{k=0}^{R_E-1} \ind_y(\Ss_k)\right).
$$
defined for $x,y \in E$, where $R_E=\inf\{k\ge 0: \, \Ss_k\in E^c\}$. We extend this function by $0$, that
is $\g_E(x,y)=0$ if $x$ or $y$ belong to $E^c$. We have a similar formula
$$
\g_E(x,y)=\g(x,y)-\EE_x(\g(\Ss_{R_E},y)),
$$
which is valid for all $x,y\in \ZZ^d$. 

We consider the $\beta$-powers of these functions:
$\gbO(x,y)$ and $\g_E^\beta$ and we shall prove similar results as in Section \ref{sec:1}.

\subsection{Cubic open sets}
In this section we consider the following family of simple open sets: {\it cubic open sets}, 
which are constructed as follows. Fix a positive integer
$m$ and a finite set $E\subset \ZZ^d$ and consider the following set
$$
Q=\bigcup\limits_{k\in E} \overline{B_\infty(k\sqrt{d/m}),1/2 \sqrt{d/m})}.
$$
We take $\O=int(Q)$ and call the cubic open set (CO)
with height $m$ and basis $E$. We point out that a CO set $\O$  can be described using 
different couples $(m,E)$, as we will see. 
These open sets are bounded and regular for the Brownian Motion 
(they satisfy for example the exterior cone condition). In general $\O$ is not connected, but a finite
union of connected components, which are also CO of the same height.

Now, fix $\O$ a CO with height $m\ge 1$ and basis $E$. 
For every $\ell\in\NN$, 
we denote by $n=n_\ell=m\, 3^{2\ell}$ and consider the set $A_n=\sqrt{d/n}\, \ZZ^d\, \cap\, \O\subset \Z$.
We denote by $E_n=\sqrt{n/d}\, A_n\subset \ZZ^d$. We point out that $\O$ is also  a CO 
with height $n$ and basis $E_n$. Each point $z\in A_n$ generates
$3^d$ points in $A_{n+1}: z+\frac13 \sqrt{d/n}\, \mathbf{e}$, with $\mathbf{e}\in \E=\{-1,0,1\}^d$. Also, 
$E_{n+1}=3E_n+\E$.

Recall that for $x\in \RR^d$, we choose $x(n)\in \Z=\sqrt{d/n}\, \ZZ^d$ any of the closets elements in $\Z$ to $x$. 
To avoid any ambiguity, if there are more than one, we take $x(n)$ the smallest such elements in the lexicographical 
order. Notice that if $x\in \O$ then $x(n)\in A_n$.

Similar to Section \ref{sec:1} we define for all $F\in \CK(\O)$ and all $x\in \RR^d$
{\footnotesize
$$
\G^{n,(\beta)}_\O (F)(x)=\sum\limits_{w\in A_n} F(w+x-x(n)) \left(\frac{d}{n}\right)^{d/2} 
\left[\g_{E_n}\left(x(n)\sqrt{n/d},w\sqrt{n/d}\right)\frac{n^{d/2-1}}{d^{d/2}}\right]^\beta.
$$
}When $\beta=1$, we denote $\G^{n,(1)}_\O=\G^{n}_\O$. Notice that $\G^{n}_\O F(x)$, for $x\in A_n$, is just
the Green potential of $F$, for the normalized random walk $(\S_{t,n})_t$ killed when exiting $\O$ 
(which is exactly exiting $A_n$)
starting from $x$. 

We observe that for $x,w\in A_n$
$$
\begin{array}{l}
\g_{E_n}\left(x\sqrt{n/d},w\sqrt{n/d}\right)\frac{n^{d/2-1}}{d^{d/2}}=\\
\\
=\g\left(x\sqrt{n/d},w\sqrt{n/d}\right)\frac{n^{d/2-1}}{d^{d/2}}
-\EE_{x\sqrt{n/d}}\left(\g\left(\Ss_{R_{E_n}},w\sqrt{n/d}\right)\frac{n^{d/2-1}}{d^{d/2}}\right)\\
\\
=\g\left(x\sqrt{n/d},w\sqrt{n/d}\right)\frac{n^{d/2-1}}{d^{d/2}}
-\EE_x\left(\g\left(\S_{T_\O,n}\sqrt{n/d},w\sqrt{n/d}\right)\frac{n^{d/2-1}}{d^{d/2}}\right),
\end{array}
$$
where the normalized random walk starts from $x$, and $T_\O=T_\O(\S_{\bullet,n})$ is the exiting
time from $\O$. 

For any $x,w$ we have $\g(x\sqrt{n/d},w\sqrt{n/d})\frac{n^{d/2-1}}{d^{d/2}}\to g(x,w)$, as $n=m3^{2\ell}$
converges to infinity. So, in order to prove that $\G^{n,(\beta)}_\O (F)(x)\to \GbO(F)(x)$, we need the following
Lemma, which is a consequence of the weak convergence of $(\S_{t,n})_t$ to $(B_t)_t$.

\begin{lemma} 
\label{lem:4}
For any $x,y\in \O$ we have
$$
\lim\limits_{n\to\infty} \EE_{x(n)}\left(\g\left(\S_{T_\O,n}\sqrt{n/d},y(n)\sqrt{n/d}\right)\frac{n^{d/2-1}}{d^{d/2}}\right)=
\EE_x(g(B_{\TO},y)).
$$
\end{lemma}
\begin{proof} 
From Lemma \ref{lem:1}, it is enough to show that
$$
\lim\limits_{n\to\infty} \EE_{x(n)}\left(g(\S_{T_\O,n},y)\right)=
\EE_x(g(B_{\TO},y)).
$$
The function $h:\partial\O\to \RR_+$ defined by $h(z)=g(z,y)$ is continuous and bounded, since
$d(y,\partial \O)>0$. In order to apply the weak convergence of $(\S_{t,n})_t$ to $(B_t)_t$, we fix 
$t_0>0$ and consider, for every $\e>0$, a continuous and bounded function $\psi:\RR_+\to [0,1]$
such that $\psi(t)=1$ for all $t\le t_0$ and $\psi(t)=0$ for all $t>t_0+\e$.
Then, 
$$
\lim\limits_{n\to\infty} \EE_{x(n)}\left(g\left(\S_{T_\O,n},y\right)\psi(\TO)\right)=
\EE_x(g(B_{\TO},y)\psi(\TO)).
$$
Here, we have used that $g(B_{\TO},y)\psi(\TO)$ is $\F_{t_0+\e}$-measurable, bounded and $B$-continuous
($\O$ is regular so $\TO(B)=T_{\overline\O}(B)$ for $\PP_x$-a.s.). This implies that
$$
\limsup\limits_{n\to\infty} \EE_{x(n)}\left(g(\S_{T_\O,n},y)\ind_{\TO\le t_0}\right)\le
\EE_x(g(B_{\TO},y)\ind_{\TO\le t_0}).
$$
Similarly, we show that
$$
\EE_x(g(B_{\TO},y)\ind_{\TO<t_0})
\le\liminf\limits_{n\to\infty} \EE_{x(n)}\left(g(\S_{T_\O,n},y)\ind_{\TO\le t_0}\right).
$$
If we take $t_0$ a continuity point for $\TO(B)$ we conclude that
$$
\begin{array}{l}
\lim\limits_{n\to\infty} \EE_{x(n)}\left(g(\S_{T_\O,n},y)\ind_{\TO\le t_0}\right)=
\EE_x(g(B_{\TO},y)\ind_{\TO\le t_0}), \hbox{ and}\\
\lim\limits_{n\to\infty} \EE_{x(n)}\left(g(\S_{T_\O,n},y)\ind_{\TO> t_0}\right)=
\EE_x(g(B_{\TO},y)\ind_{\TO> t_0}).
\end{array}
$$
This finishes the proof.
\end{proof}

It is straightforward  to generalize Lemma \ref{lem:2} and Proposition \ref{pro:2} to the present setting. Also
notice that $\g_E\le \g$ for any set $E$, then we can use the bounds developed in Section \ref{sec:1} (in particular
\eqref{eq:5}, \eqref{eq:2}) to prove the following result.

\begin{proposition} 
\label{pro:2.1} Let $\O$ be a CO, then 
for all $\beta \in [1,\frac{d}{d-2})$, $x\in \O$ and all $F\in \CK(\O)$ the following limits exist
\begin{enumerate}[(i)]
\item $\lim\limits_{n\to \infty} \G^{n,(\beta)}_\O (F)(x)=G^{(\beta)}_\O (F)(x)$;

\item $0\le \lim\limits_{n\to \infty} \int \left(\G^{n,(\beta)}_\O (F)(x)-1\right)^+ F(x) \,dx=
\int \left(G^{(\beta)}_\O (F)(x)-1\right)^+ F(x)\, dx$.
\end{enumerate}
\end{proposition} 

\subsection{General bounded regular open sets}

The purpose of this section is to generalize Proposition \ref{pro:2.1} to the case of a general bounded regular
open set $\O\subset \RR^d$. To this end, we take the following approximations of $\O$. 
For every positive $m$ consider
$$
\begin{array}{l}
A^i_{m}=\{x\in \sqrt{d/m}\,  \ZZ^d: d_\infty(x,\O^c)>\sqrt{d/m}\};\\
A^e_{m}=\{x\in \sqrt{d/m}\,  \ZZ^d: d_\infty(x,\O)<\sqrt{d/m}\}.
\end{array}
$$
We assume that $m_0$ is large enough so $A^i_{m_0}$ is not empty, and consider only integers
of the form $m=m_\ell=m_0 3^{2\ell}$.
Notice that $A^i_m\subset A^e_m$ and
$A^i_m\subset A^i_n$ if $m\le n$. 
We define the CO sets
$$
\begin{array}{l}
\O_m^i=int\left(\cup_{x\in A^i_{m}} \overline{B_\infty(x,1/2\sqrt{d/m})}\right)\subset \O;\\
\O_m^e=int\left(\cup_{x\in A^e_{m}} \overline{B_\infty(x,1/2\sqrt{d/m})}\right)\supset \overline\O,
\end{array}
$$
where here $i,e$ means interior and exterior respectively (similarly for $A^i_m,A^e_m$). We also 
denote by $I_m=\sqrt{m/d}\, A^i_m, E_m=\sqrt{m/d}\, A^e_m$. We point out that 
$\O_m^i\subset \O_{m+1}^i\subset \O_{n+1}^e\subset \O_{n}^e$, for all $m,n$.

\smallskip

For any fixed $m$, Lemma \ref{lem:4} can be applied to 
$\O_m^e, \O_m^i$. For $x,y \in \O_m^i$ and  $m\le n$, we obviously have
$$
\G^n_{\O_m^i} (\ind_{\{y(n)\}})(x(n))\le \G^n_{\O_n^e} (\ind_{\{y(n)\}})(x(n))\le 
\G^n_{\O_m^e} (\ind_{\{y(n)\}})(x(n))
$$
which gives
$$
\EE_{x(n)}\left(\g\left(\S_{T_{\O_m^i},n}\sqrt{n/d}\, ,y(n)\sqrt{n/d}\right)\right)\ge 
\EE_{x(n)}\left(\g\left(\S_{T_{\O_n^e},n}\sqrt{n/d}\, ,y(n)\sqrt{n/d}\right)\right).
$$
Thus,
$$
\begin{array}{ll}
\EE_{x}(g(B_{T_{\O_m^i}},y))\hspace{-0.3cm}&=
\lim\limits_{n\to\infty} \frac{n^{d/2-1}}{d^{d/2}} 
\EE_{x(n)}\left(\g\left(\S_{T_{\O_m^i},n}\,\sqrt{n/d}\, ,y(n)\sqrt{n/d}\right)\right)\\
\\
\hspace{-0.3cm}&\ge \limsup\limits_{n\to\infty} 
\frac{n^{d/2-1}}{d^{d/2}}\EE_{x(n)}\left(\g\left(\S_{T_{\O_n^e},n}\,\sqrt{n/d}\, ,y(n)\sqrt{n/d}\right)\right)\\
\\
\hspace{-0.3cm}&\ge \liminf\limits_{n\to\infty} 
\frac{n^{d/2-1}}{d^{d/2}}\EE_{x(n)}\left(\g\left(\S_{T_{\O_n^e},n}\, \sqrt{n/d}\, ,y(n)\sqrt{n/d}\right)\right)\\
\\
\hspace{-0.3cm}&\ge \lim\limits_{n\to\infty} 
\frac{n^{d/2-1}}{d^{d/2}}\EE_{x(n)}\left(\g\left(\S_{T_{\O_m^e},n}\, \sqrt{n/d}\, ,y(n)\sqrt{n/d}\right)\right)\\
\\
\hspace{-0.3cm}&=\EE_{x}(g(B_{T_{\O_m^e}},y)).
\end{array}
$$
Since $\O_m^i\uparrow \O$ and $\O_m^e\downarrow \overline \O$, as $m\to \infty$, 
then the regularity of $\O$ is used to show the analogous of Lemma \ref{lem:4}: For all $x,y \in \O$, with $a=e,i$
\begin{equation}
\label{eq:4.5}
\lim\limits_{n\to\infty} \frac{n^{d/2-1}}{d^{d/2}} \EE_{x(n)}\left(\g\Big(\S_{T_{\O_n^a},n}
\sqrt{n/d},y(n)\sqrt{n/d}\Big)\right)=\EE_x(g(B_\TO,y)).
\end{equation}
We have essentially shown the following result.
\begin{proposition} 
\label{pro:2.2} Let $\O$ be a bounded regular open set. Then,
for all $\beta \in [1,\frac{d}{d-2})$, $x\in \O$ and all $F\in \C_\K(\O)$ the following limits hold
\begin{enumerate}[(i)]
\item $\lim\limits_{n\to \infty} \G^{n,(\beta)}_{\O_n^e} (F)(x)=G^{(\beta)}_\O (F)(x)$;

\item $0\le \lim\limits_{n\to \infty} \int_\O \left(\G^{n,(\beta)}_{\O_n^e} (F)(z)-1\right)^+ F(z) \,dz=
\int _\O\left(G^{(\beta)}_\O (F)(z)-1\right)^+ F(z)\, dz$.
\end{enumerate}
\end{proposition} 

The positive operator $G^{(\beta)}_\O$ can be extended to the set $\BB_b=\BB_b(\O)$ 
of bounded measurable functions, because
if $f\in \BB_b$, we have
\begin{equation}
\label{eq:6}
\begin{array}{ll}
|G^{(\beta)}_\O (f)(x)|\le \int_\O |f(y)| \gbO (x,y) \, dy
\le \|f\|_\infty \int_\O g^\beta (x,y) \, dy\\
\\
\le \|f\|_\infty \int_{B(0,R)} g^\beta (0,y) \, dy
=D(d,\beta,\O)\|f\|_\infty,
 \end{array}
\end{equation}
where $D(d,\beta,\O)=\left(\frac{\Gamma(d/2-1)}{2\pi^{d/2}}\right)^\beta S(d) \frac{R^{d-\beta(d-2)}}{{d-\beta(d-2)}}$,
with $S(d)$ the surface of the (Euclidean) unit ball in $\RR^d$ and $R=diameter(\O)$. 

It is also strightforward to show that $\GbO$ maps $\BB_b$ into $\C_0(\O)$. 
Indeed, take any $x\in \O$, $\e>0$ small enough
and $(x_n)_n\subset \O$ such that $x_n\to x$. Let us start with
$$
\GbO(f)(x_n)=\int_{B(x,\e)}f(y) \gbO(x_n,y)\, dy +\int_{y: d(y,x)>\e} f(y) \gbO(x_n,y)\, dy.
$$
For every fixed $\e$ the second integral converges to $\int_{y: d(y,x)>\e} f(y) \gbO(x,y)\, dy$.

On the other hand $|\int_{B(x,\e)}f(y) \gbO(x_n,y)\, dy|\le \|f\|_\infty \int_{B(x,\e)} g^\beta (x_n,y)\,dy$, which
is uniformly bounded by $\|f\|_\infty \int_{B(x,2\e)} g^\beta (x,y)\,dy$, if $\|x-x_n\|\le \e$. This last integral
converges to $0$ when $\e\downarrow 0$,  proving
that $\GbO(f)$ is a continuous function on $\O$. We already know it is also a bounded function. In particular
$G^{(\beta)}_\O:\C_b(\O)\to \C_b(\O)$ is a bounded positive linear operator.
Now, let us show that $\GbO(f)\in \C_0(\O)$, for any $f\in \BB_b$. For that purpose, we notice that 
$$
|\limsup\limits_{z\to \partial \O} \GbO(f)(z)|\le \|f\| \limsup\limits_{z\to \partial \O} \int_\O \gbO(z,y) \, dy,
$$ 
which is finite (it suffices to replace $\gbO$ by $g^\beta$). 
So, it is enough to show that $a=\limsup\limits_{z\to \partial \O} \int_\O \gbO(z,y) \, dy$ is zero.
We take any sequence 
$(z_k)_k \subset \O$ such that $z_k \to \partial \O$ and
$$
a=\lim\limits_{k \to \infty} \int_\O \gbO(z_{k},y) \, dy.
$$
Since, $\O$ is bounded we can assume further that $z_k \to z\in \partial \O$.
The regularity of $\O$ shows that for all $y\in \O$ we have $ g^\beta_\O(z_{k},y)\to 0$, with $k\to \infty$.

For $\e>0$, we have 
$$
\int_\O \gbO(z_{k},y) \, dy \le \int_{\O\cap  \|z-y\|\le 2\e} \gbO(z_{k},y) \, dy+
\int_{\O\cap  \|z-y\|>2 \e} \gbO(z_{k},y) \, dy.
$$
The second term converges to $0$ from the Dominated Convergence Theorem. Indeed, 
the regularity of $\O$ shows that for all $y\in \O$ we have $ g^\beta_\O(z_{k},y)\to 0$, when $k\to \infty$.
On the other hand, for large $k$ we can assume that $\|z_k-y\|>\e$ and then
$\gbO(z_{k},y)\le g^\beta(z_{k},y)\le C(d)^\beta (\e)^{\beta(2-d)}$, providing the desired domination.

For the first term, if $k$ is large enough such that $\|z-z_k\|\le \e$, we get 
$$
\int_{\O\cap  \|z-y\|\le 2\e} \gbO(z_{k},y) \, dy\le \int_{\O\cap  \|z-y\|\le 2\e} g^\beta(z_{k},y) \, dy
\le \int_{ \|y\|\le 3\e} g^\beta(0,y) \, dy.
$$
We conclude that $a\le \lim\limits_{\e\downarrow 0} \int_{ \|y\|\le 3\e} g^\beta(0,y) \, dy=0$,
and the claim is shown. 

\medskip

So far we have proved that $\GbO:\BB_b(\O)\to \BB_b(\O)$ is a bounded positive linear operator, which also
satisfies $\GbO:\BB_b(\O)\subset \C_0(\O)$. We now extend $(ii)$ in Proposition \ref{pro:2.2} for $f\in \BB_b(\O)$,
that is $\GbO$ satisfies the CMP in $\BB_b(\O)$.
\begin{lemma}
\label{lem:2.4}
Assume $f\in \BB_b(\O)$ then
$$
0\le \int_\O \left(G^{(\beta)}_\O (f)(x)-1\right)^+ f(x)\, dx.
$$
\end{lemma}
\begin{proof} Similar to \eqref{eq:6}, for any $f\in \BB_b(\O)$ and every $\e>0$, we have
$$
|\GbO(f)(x)|\le \|f\|_\infty \int_{B(0,\e)} g^\beta(0,y)+\left(\frac{C(d)}{\e^{d-2}}\right)^\beta \|f\|_1.
$$
So, if $(F_k)_k\subset \CK(\O)$ is a sequence of functions such that $\|F_k-f\|_1\to 0$
and $\|F_k\|_\infty \le \|f\|_\infty$, then
$$
\|\GbO(f)-\GbO(F_k)\|_\infty\le 2\|f\|_\infty \int_{B(0,\e)} g^\beta(0,y)+
\left(\frac{C(d)}{\e^{d-2}}\right)^\beta \|F_k-f\|_1.
$$
Therefore, taking limits in $k$ and then in $\e$, we conclude that 
$(\GbO(F_k))_k$ converges uniformly to $\GbO(f)$. The rest of the proof is an application of 
Dominated Convergence Theorem
and Proposition \ref{pro:2.2} $(ii)$.
\end{proof}

In summary, $\GbO:\BB_b(\O)\to \BB_b(\O)$ is a positive, bounded  linear operator 
and from the previous Lemma it satisfies the CMP in $\BB_b(\O)$. Hence,
$\GbO$ is the potential of a unique contraction resolvent
$\U=(U^\lambda)_{\lambda\ge 0}$, defined on $\BB_b(\O)$. In particular, 
for each $\lambda$ we have $U^\lambda:\BB_b(\O)\to \BB_b(\O)$ is a bounded
positive linear operator such that $\|\lambda U^\lambda\|\le 1$ and
$$
U^0=\GbO.
$$
See for example Lemma 4.1.9 in \cite{MarcusRosen2006} and Remark 4.1.10 applied to the
Banach space $(\BB_b(\O),\|\,\|_\infty)$. We recall that $G^{(\beta)}_\O(\BB_b(\O))\subset 
\C_0(\O)$, and the way $\U$ is constructed implies that for all $\lambda$ one has
$U^\lambda(\BB_b(\O))\subset \C_0(\O)$ (recall that $U^\lambda=\GbO(\II+\lambda \GbO)^{-1}$).
We prove now an extra property of this resolvent.

\begin{proposition} 
\label{pro:2.5}
The resolvent $\U$ is continuous on $\lambda$, that is,
for all $\lambda\ge 0$
$$
\lim\limits_{\lambda'\to \lambda} \|U^\lambda-U^{\lambda'}\!\|=0.
$$
In case $\lambda=0$, the limit is taken as $\lambda'>0$. In particular, for all $f\in \BB_b(\O)$ we have
$$
\|U^\lambda(f)-\GbO(f)\|_\infty \to 0,
$$
as $\lambda\downarrow 0$.
\end{proposition} 
\begin{proof}
We start with the case $\lambda=0$. The resolvent equation shows that
$U^{\lambda'}\le U^0$ on $\BB_b(\O)$. Using inequality \eqref{eq:6} and the resolvent equation, 
we get for all $x\in \O$
$$
|U^0(f)(x)-U^{\lambda'}(f)(x)|\le \lambda' U^{\lambda'} (U^0(|f|))(x)
\le \lambda' U^{0} (U^0(|f|))(x)\le \lambda' D^2 \|f\|_\infty,
$$
and the result is shown in this case. 

For $\lambda>0$, the result is a consequence of the resolvent equation and the fact that the resolvent
is a contraction
$$
\|U^\lambda-U^{\lambda'}\|\le \frac{|\lambda-\lambda'|}{\lambda\lambda'}\|\lambda U^\lambda\| 
\|\lambda' U^{\lambda'}\| \le \frac{|\lambda-\lambda'|}{\lambda\lambda'}.
$$
\end{proof}

Our next step is to show that this resolvent comes from a Ray process (see Chapter 4 in 
\cite{MarcusRosen2006}).
To this end, we recall that a function $f\in \BB_b^+$ is supermedian if 
$$
\alpha U^\alpha (f) (x)\le f(x),
$$
for all $x\in \O$ and some (all) $\alpha>0$. We denote by $\mathcal{M}^+$ the set of supermedian
functions and $\mathcal{M}=\mathcal{M}^+-\mathcal{M}^+$ the linear space generated. It is well known
that $\mathcal{M}^+$ is closed under monotone pointwise convergence,  $\mathcal{M}$ is a lattice
and contains the constants.  Using the resolvent equation it is straightforward to show that
$\GbO(\BB_b^+)\subset \mathcal{M}^+$.
The following technical result is needed.

\begin{lemma} 
\label{lem:3}
The set $\mathcal{H}=\mathcal{M}\cap \C_0(\O)$ is dense in $\C_0(\O)$.
\end{lemma}
\begin{proof} Recall that  $\hat\O=\O\cup\{\partial\}$ is the one point compactification of $\O$, and that
$\C(\hat\O)$ is identified with $\C_0(\O)\oplus \ind$. The set
$\hat{\mathcal{H}}=\mathcal{M}\cap \C(\hat\O)$ is a lattice that contains the constants and according to the 
Stone$\,$-Weierstrass Theorem is dense in $\C(\hat\O)$, as soon as it separates points. 

Take first, $x,y \in \O$
and consider $f_\e=f=\ind_{B(x,\e)}$ for small $0<\e<1$ to be determined later. At least we assume for the moment that
$B(x,2\e)\subset \O$ and  $\e<\|x-y\|$.  The function $F=G^{(\beta)}_\O (f)$ belongs to 
$\mathcal{M}^+\cap \C_0(\O)$ and we have for 
$z\in B(x,\e), \xi\in \partial\O$
\begin{eqnarray}
&g(\xi,z)\hspace{-0.3cm}&=C(d)\|\xi-z\|^{2-d}\le d(z,\partial\O)^{2-d}\|x-z\|^{d-2}g(x,z)\nonumber\\
\nonumber\\
&\hspace{-0.3cm}&\le d(z,\partial\O)^{2-d}\e^{d-2}g(x,z)
\le  (d(x,\partial\O)-\e)^{2-d}\e^{d-2}g(x,z)\label{eq:domination}\\
\nonumber\\
&\hspace{-0.3cm}&\le (d(x,\partial\O))^{2-d}(2\e)^{d-2}g(x,z),\nonumber
\end{eqnarray}
and therefore $\gbO(x,z)\ge g^\beta(x,z)\left(1-d(x,\partial\O)^{2-d}(2\e)^{d-2}\right)^{\beta}$. On the other hand,
$\gbO(y,z)\le g^\beta(y,z)=C^\beta(d)\|y-z\|^{\beta(2-d)}$, moreover 
$$
\|y-z\|\ge \|y-x\|-\e\ge \frac{(\|y-x\|-\e)}{\e}\|x-z\|,
$$ 
and we conclude
$$
\gbO(y,z)\le (\|y-x\|-\e)^{\beta(2-d)} \e^{\beta(d-2)}\, g^{\beta}(x,z).
$$
Thus,
$$
\begin{array}{ll}
F(y)=G^{(\beta)}_\O (f)(y)\hspace{-0.3cm}&\le (\|y-x\|-\e)^{\beta(2-d)} \e^{\beta(d-2)} \int\limits_{B(x,\e)} g^{\beta}(x,z) \, dz\\
\hspace{-0.3cm}&\le \left(\frac{(\|y-x\|-\e)^{2-d}}{1-d(x,\partial\O)^{2-d}(2\e)^{d-2}}\right)^\beta 
\e^{\beta(d-2)}\, G^{(\beta)}_\O (f)(x).
\end{array}
$$ 
By taking $\e$ small enough, we prove that  $G^{(\beta)}_\O (f)(y)<G^{(\beta)}_\O (f)(x)$. The function
$F\in \mathcal{M}^+\cap \C_0(\O)$ separates the points $x,y$.

Now, we separate  $x\in\O$ and $\partial$. For this we need the constant functions. If we
take now $F=G^{(\beta)}_\O (\ind)+\ind$. Clearly we have $F(x)>1=F(\partial)$.

So far we have shown that $\mathcal{M}\cap \C(\hat\O)$ is dense in $\C(\hat\O)$. In particular is dense
in $\C_0(\O)$. So, for any $F\in \C_0(\O)$ there exists a sequence $(F_n)_n\subset \mathcal{M}\cap \C(\hat\O)$, 
which converges uniformly on $\hat\O$ to $F$ (extended by $0$ on $\partial$). The sequence defined by
$H_n(x)=F_n(x)-F_n(\partial)\in \mathcal{M}\cap \C_0(\O)$ converges uniformly to $F$ and the result is shown.
\end{proof}

We can apply now Theorem 4.7.1 in \cite{MarcusRosen2006} to show the following result, which is closer to
Theorem \ref{the:1}.

\begin{theorem} 
\label{the:2.7}
Assume that $\O\subset \RR^d$ is a bounded regular open set. Then, for all
$\beta\in [1,\frac{d}{d-2})$ there exists a sub-Markov Ray-semigroup $\PS=(P_t)_t$, defined on 
$\BB_b(\O)$, such that
\begin{enumerate}[(i)]
\item for all $F\in \C_b(\O)$ the function $(t,x)\to P_t(F)(x)$ is, for every $x$, right continuous in $t\in [0,\infty)$ and
for every $t$, a Borel measurable function on $x$. Therefore, it is jointly a Borel measurable function.

\item for all $F\in \C_b(\O), x \in \O, \lambda\ge 0$ 
$$
U^\lambda(F)(x)=\int_0^\infty e^{-\lambda t}\, P_t(F)(x)\, dt.
$$
\end{enumerate}
\end{theorem}
\begin{proof} Theorem 4.7.1 in \cite{MarcusRosen2006} shows 
$(i)$ and $(ii)$ except for the case $\lambda=0$, which we now prove.
Proposition \ref{pro:2.5} and the monotone convergence Theorem shows that for all
$F\in \C_b(\O)^+$ it holds
$$
\begin{array}{ll}
\GbO(F)(x)\hspace{-0.3cm}&=U^0(F)(x)=\lim\limits_{\lambda \downarrow 0} U^\lambda(F)(x)=
\lim\limits_{\lambda \downarrow 0} \int_0^\infty e^{-\lambda t}\, P_t(F)(x)\, dt\\
\hspace{-0.3cm}&=\int_0^\infty P_t(F)(x)\, dt.
\end{array}
$$
\end{proof}
Recall that a sub-Markov Ray-semigroup $\PS=(P_t)_t$ is like a sub-Markov semigroup, except that
$P_0$ may not be the identity. So, 
\begin{enumerate}
\item for all $t,s \in \RR^+, f\in \BB_b, x \in \O$ we have $P_{t+s}(f)(x)=P_t(P_s(f))(x)=P_s(P_t(f))(x)$;
\item each $P_t$ is a positive operator and $P_t (\ind) \le \ind$, so it is a contraction. Thus, for all
$t\ge 0$, we can decompose $P_t(f)(x) =\int f(y) P_t(x,dy)$, 
where $P_t(x,dy)$ is a sub-probability measure, for $(t,x)$ fixed, which is measurable in $(t,x)$;
\item for all $F\in \C_b(\O), x \in \O$ the function $t\to P_t(F)(x)$ is right continuous in $[0,\infty)$, and
if $f\in \mathcal{M}^+$ then $P_t(f)(x)$ is decreasing in $t$. 
\end{enumerate}
The fact that $P_0$ is not the identity has important consequences. For example, 
$P_t(F)(x)$ is, in general, a discontinuous function of $x$, 
even if $F\in \CK(\O)$.

In \cite{MarcusRosen2006}, a Ray process $X$ is constructed taking values in $\hat \O$ with c\`adl\`ag paths 
and associated semigroup an extension of  $(P_t)_t$. 
The semigroup is extended to $(\hat \O, \B(\hat \O))$ by simply putting  
$\P_t(x,\partial)=1-P_t(x,\O)$
and $\P_t(\partial,\partial)=1$. Similarly, we extend the resolvent $(\overline{U}^\lambda)_{\lambda\ge 0}$ and for any
function defined on $\O$ we set $\overline f$ the extension to $\hat \O$ given by $\overline f(\partial)=0$.
Notice that $\overline f \in \C_b(\hat \O)$ iff $f\in \C_0(\O)$. We also have
$$
\P_t(\overline f)(x)=\overline{P_t(f)}(x),\quad \overline{U}^\lambda(\overline f)(x)=\overline{U^\lambda(f)}(x).
$$
We remark that $\overline{U}^\lambda(\overline f)$ is in general not continuous at $\partial$. 
On the other hand, for all $f\in \BB_b(\hat \O)$ and $t,s\ge 0$ it holds
$$
\EE(f(X_{t+s})\big| \F_t)=\P_s(f)(X_t).
$$

For what it follows, a distinguished set is the set of branching points denoted by 
$$
\mathcal{N}=\{x\in \O: P_0(x,dy)\neq \delta_x(dy)\},
$$
and we put $D=\O\setminus \mathcal{N}$. We denote by $\widehat D=D\cup \partial$. It is known that
\begin{enumerate}
\item $D$ is a Borel set;
\item for all $x\in \widehat D$ we have $P_0(x,dy)=\delta_x(dy)$;
\item for all $t\ge 0$ and all $x\in \O$, we have  $\P_t(x,\mathcal{N})=0$.
\end{enumerate}

An interesting result (see Lemma 4.7.9) is that $X_t(\omega)\in \widehat D$ for all $t\ge 0$, $\PP$-a.s. 
Nevertheless, at some times $t$ the left limit $X_{t-}(\omega)$ may belong to $\mathcal{N}$, 
where the process branches again. 

In what follows, we prove that $X$ is a Feller
process. This is equivalent to show that $\mathcal{N}$ is empty, and this will finish the proof of Theorem \ref{the:1}, 
when $\O$ is a bounded regular open set.

\begin{lemma} 
\label{lem:2.8}
$\mathcal{N}=\varnothing$.
\end{lemma}
\begin{proof} Consider $F\in \C_b(\O)^+$. Since $\GbO(F)(x)=\int P_t(F)(x) \, dt$, Fubini's theorem and 
the semigroup property shows for all $x$
$$
\begin{array}{l}
P_0(\GbO(F))(x)=\int \GbO(F)(y) P_0(x,dy)=
\int \int_0^\infty P_t(F)(y)\, dt \, P_0(x,dy)\\
=\int_0^\infty \int P_t(F)(y)\, P_0(x,dy)\, dt
= \int_0^\infty P_t(F)(x)\, dt=\GbO(F)(x).
\end{array}
$$
Hence, $P_0(\GbO(F))=\GbO(F)$, for all $F\in \C_b(\O)$. Another way to say this is, if
we take $H\in \GbO(\C_b(\O))$, then $P_0(H)(x)=H(x)$ holds for all $x\in \O$.
From here we cannot conclude that $P_0(x,dz)=\delta_x(dz)$, unless we can prove that 
$\GbO(\C_b(\O))$ is dense in $\C_0(\O)$, which turns out to be equivalent to the result we are trying to prove.

Let us continue with the proof. For $x\in \O$, we consider, $f=\ind_{\overline B}$, 
where $B=B(x,\e)$. We assume that $\e>0$ is small enough such that
$\overline{B}(x,2\e)\subset \O$. We define $F=\GbO(f)$.
Take a decreasing sequence $(H_k)_k$ of continuous
functions with compact support in $\overline{B}(x,2\e)$, 
taking values on $[0,1]$, such that $H_k\downarrow f$. Since $\GbO(H_k)\downarrow \GbO(f)$
and $P_0(x,\bullet)$ is a finite measure, we conclude from the monotone convergence Theorem
$P_0(F)=P_0(\GbO(f))=\GbO(f)=F$, which is one of the main ingredients we need.

\medskip

We obtain the following lower estimate (see \eqref{eq:domination} in the proof of Lemma \ref{lem:3})
\begin{equation}
\label{eq:7}
\begin{array}{ll}
P_0(F)(x)\hspace{-0.3cm}&=F(x)\ge \left(1-d(x,\partial\O)^{2-d}(2\e)^{d-2}\right)^{\beta} \int_{B} g^\beta(x,z)\, dz\\
\\
\hspace{-0.3cm}&=\left(1-d(x,\partial\O)^{2-d}(2\e)^{d-2}\right)^{\beta} G^{(\beta)}(f)(x).
\end{array}
\end{equation}
On the other hand, $F(z)=\GbO(f)(z)\le \int_B g^\beta(z,w)\, dw=G^{(\beta)}(f)(z)$. It is straightforward
to show that $G^{(\beta)}(f)(z)\le G^{(\beta)}(f)(x)$. When 
$0<\|z-x\|\le 2\e$  this property can be shown
using a reflection with respect to the hyperplane with normal $(x-z)/\|x-z\|$ passing through 
$\frac12(z+x)$ (this is just the reflection principle for BM). Indeed, take the regions 
$$
R_1=\{w\in B: \|w-z\|\le \|w-x\|\},
$$ 
$R_2$ the reflection of $R_1$ with respect to the hyperplane and 
$R_3=B\setminus (R_1\cup R_2)$. Then
$$
\begin{array}{l}
\int_{R_1} g^\beta(z,w) \, dw=\int_{R_2} g^\beta(x,w) \, dw\\
\\
\int_{R_2} g^\beta(z,w) \, dw=\int_{R_1} g^\beta(x,w) \, dw\\
\\
\int_{R_3} g^\beta(z,w) \, dw<\int_{R_3} g^\beta(x,w) \, dw,
\end{array}
$$
and the claim follows.

Now, if $k\ge 2$ and  $\|z-x\|> k\e$, we obtain $\|z-w\|\ge (k-1)\|x-w\|$, for any $w\in B$, showing that
$$
F(z)\le G^{(\beta)}(f)(z)\le (k-1)^{\beta(2-d)} G^{(\beta)}(f)(x)\le G^{(\beta)}(f)(x)
$$
So, for fix $k\ge 3$, we get
$$
\begin{array}{l}
P_0(F)(x)=\int_{B(x,k\e)} F(z) P_0(x,dz)+\int_{\O\cap (B(x,k\e))^c} F(z) P_0(x,dz)\\
\\
\le G^{(\beta)}(f)(x)\left(\int_{B(x,k\e)} P_0(x,dz)
+(k-1)^{\beta(2-d)}\int_{\O\cap (B(x,k\e))^c} P_0(x,dz)\right)\\
\\
\le G^{(\beta)}(f)(x) \Big(P_0(x,B(x,k\e))+(k-1)^{\beta(2-d)}(1-P_0(x,B(x,k\e)))\Big).
\end{array}
$$
In the last inequality we have used that $P_0(x,\bullet)$ is a measure whose total mass is at most 1. Using
the lower bound obtained in \eqref{eq:7} we conclude that
$$
\left(1-d(x,\partial\O)^{2-d}(2\e)^{d-2}\right)^{\beta}\le P_0(x,B(x,k\e))+(k-1)^{\beta(2-d)}(1-P_0(x,B(x,k\e)))
$$
and taking $\e\downarrow 0$ yields
$$
1\le P_0(x,\{x\})(1-(k-1)^{\beta(2-d)})+(k-1)^{\beta(2-d)},
$$
which is possible only if $P_0(x,\{x\})=1$ and therefore $x\notin \mathcal{N}$. The result is shown.
\end{proof}
This finishes the proof of Theorem \ref{the:1}, when $\O$ is a bounded regular open set. 

\section{Powers of the Green potential in unbounded regular open sets}
\label{sec:3}
In this section we shall prove Theorem \ref{the:1} for $\O$ an unbounded regular open set. We shall use 
the same notation of previous sections. 
The first thing we shall prove is that $\GbO$ satisfies the CMP on $\CK$. For that purpose, we approximate
$\O$ by an increasing sequence of bounded regular open sets. For every
$n\ge 1$ we define $\O_n=\O\cap B(0,n)$, and we assume $n$ is large enough, so $\O_n$ is not empty.
It is straightforward to show that $\O_n$ is also regular. 

For each $n$ consider $\GbOn$, which is a positive bounded linear operator defined on $\C_b(\O_n)$
that satisfies there the CMP. Moreover, for every $F\in \C_b(\O_n)$ we have
$$
\int_{\O_n} \left(\GbOn(F)(x)-1\right)^+\, F(x)\, dx\ge 0.
$$
The Green kernel of $\GbOn$ is $g_{\O_n}^\beta(x,y)=\left(g(x,y)-\EE_x(g(B_{T_{\O_n}},y))\right)^\beta$, which
converges pointwise, for all $x\neq y \in \O$, to 
$g_{\O}^\beta(x,y)=\left(g(x,y)-\EE_x(g(B_{T_{\O}},y))\right)^\beta$, as $n\to \infty$.

Consider, now $F\in \CK(\O)$ and $n$ large enough such that $\supp F\subset \O_n$. Then, for $\e>0$ 
we have
$$
\GbOn(F)(x)=\int_{B(x,\e)\cap \O_n} F(y) g^{\beta}_{\O_n}(x,y)\, dy+
\int_{\O_n\setminus B(x,\e)} F(y) g^{\beta}_{\O_n}(x,y)\, dy.
$$
The first integral is bounded by 
$\|F\|_\infty \int_{B(x,\e)} g^{\beta}(x,y)\, dy=\|F\|_\infty \int_{B(0,\e)} g^{\beta}(0,y)\, dy$. 
The second integral converges, for every fixed $\e$, as $n\to \infty$ to
$$
\int_{\O\setminus B(x,\e)} F(y) g^{\beta}_{\O}(x,y)\, dy,
$$
because the Dominated Convergence Theorem. Here the domination is given by the function $(C(d) \e^{2-d})^\beta |F|$. 
Thus,
$$
\limsup\limits_{n\to \infty} |\GbOn(F)(x)-\GbO(F)(x)|\le 2\|F\|_\infty  \int_{B(0,\e)} g^{\beta}(0,y)\, dy,
$$
which converges to $0$ as $\e\downarrow 0$.
Notice that we also have the uniform domination for $\GbOn(F)(x)$ given by
$$
|\GbOn(F)(x)|\le \|F\|_\infty \int_{B(0,a)} g^{\beta}(0,y)\, dy+\left(\frac{C(d)}{a^{d-2}}\right)^\beta
\int_{\O} |F(y)|\, dy,
$$ 
which is also inherit by $\GbO$
\begin{equation}
\label{eq:3.00}
|\GbO(F)(x)|\le \|F\|_\infty \int_{B(0,a)} g^{\beta}(0,y)\, dy+\left(\frac{C(d)}{a^{d-2}}\right)^\beta \|F\|_1,
\end{equation}
valid for every $a>0$.

Again an application of the Dominated Convergence Theorem, we conclude that
$$
0\le \lim\limits_{n\to \infty} \int_{\O_n} \left(\GbOn(F)(x)-1\right)^+\, F(x)\, dx=
\int_{\O} \left(\GbO(F)(x)-1\right)^+\, F(x)\, dx.
$$
We have proved the following result.

\begin{proposition} $\GbO$ satisfies the CMP in $\CK(\O)$. Moreover, for every
$F\in \CK(\O)$ we have 
\begin{equation}
\label{eq:3.0}
0\le\int_{\O} \left(\GbO(F)(x)-1\right)^+\, F(x)\, dx.
\end{equation}
\end{proposition}

For reasons that will be clear later on, we need to extend $\GbO$ to $\BB_b(\O)\cap L^1$, and prove
it satisfies the CMP there (here $L^1=L^1(\O,dx)$). 

Inequality \eqref{eq:3.00} allow us to extend $\GbO$ to $\BB_b(\O)\cap L^1$. Indeed, 
for a fixed $f\in \BB_b(\O)\cap L^1$ consider a sequence of functions $(F_k)_k \in \CK(\O)$, such that 
$\|F_k\|_\infty\le \|f\|_\infty$ and  $\|f-F_k\|_1\to 0$. 

Then, for every $\e>0$, choose $a=a(\e)>0$ such that 
$$
2\|f\|_\infty \int_{B(0,a)} g^{\beta}(0,y)\, dy\le \e/2
$$ 
and then choose $n_0=n_0(\e)$ such that $\left(\frac{C(d)}{a^{d-2}}\right)^\beta \|F_k-F_m\|_1\le \e/2$, for all
$k,m\ge n_0$. 

Then $\|\GbO(F_k)-\GbO(F_m)\|_\infty\le \e$, which means that $\GbO$ is well defined
on $\BB_b(\O)\cap L^1$ and it satisfies inequality \eqref{eq:3.00} there.
Using the Dominated Convergence Theorem, we prove this extension also satisfies the CMP on $\BB_b(\O)\cap L^1$.

Now, we prove that for $f\in \BB_b(\O)\cap L^1$, the function $\GbO(f)\in \C_0(\O)$. 
Take $(x_n)_n\subset \O$. We consider first the case where $x_n\to x\in\O$. Fix a positive $\e$ 
and assume that $\|x-x_n\|\le \e$, then
$$
\begin{array}{ll}
|\GbO(f)(x)-\GbO(f)(x_n)|\le &2\|f\|_\infty \int_{B(0,3\e)} g^\beta(0,y)\, dy\\
&+\int_{\O\setminus B(x,2\e)} |f(y)| |\gbO(x_n,y)-\gbO(x,y)|\,dy.
\end{array}
$$
The second integral converges to $0$ as $n\to \infty$ by the Dominated Convergence Theorem. Then, by taking $\e\downarrow 0$
the first term converges to $0$, and the continuity of $\GbO(F)$ is shown in $\O$. 

Assume now $x\in \partial \O$. We need to show that $\GbO(f)(x_n)\to 0$. This is done in a similar way, by
decomposing the integral as above (here we use that $\O$ is regular). 
Finally, we assume that $\|x_n\|\to \infty$. Consider  $M>0, \e>0$ and decompose
$$
\begin{array}{l}
|\GbO(f)(x_n)|\le \|f\|_\infty \int_{\O\cap B(0,M)} \gbO(x_n,y)\, dy
+\|f\|_\infty\int_{B(x_n,\e)\cap \O} \gbO(x_n,y) dy\\ 
\\
+\int_{\O\setminus (B(0,M)\cup B(x_n,\e))} |f(y)| \gbO(x_n,y)\, dy\\
\\
\le \|f\|_\infty \int_{B(0,M)} g^{\beta}(x_n,y)\, dy+\|f\|_\infty \int_{B(0,\e)} g^\beta(0,y) dy\\
\\
+(C(d)\e^{2-d})^\beta \int_{\O\setminus B(0,M)} |f(y)|\, dy.
\end{array}
$$
For $M$ fixed, the first term converges to $0$ as $n\to \infty$, by the Dominated Convergence Theorem. Hence,
$$
\limsup\limits_{n \to \infty} |\GbO(f)(x_n)|\le \|f\|_\infty \int_{B(0,\e)} g^\beta(0,y) dy \\
+(C(d)\e^{2-d})^\beta \int_{\O\setminus B(0,M)} |f(y)|\, dy.
$$
Now, we let $M\to \infty$ to get 
$\limsup\limits_{n \to \infty} |\GbO(f)(x_n)|\le \|f\|_\infty \int_{B(0,\e)} g^\beta(0,y) dy$. Finally,
the claim follows by making $\e\downarrow 0$.

\bigskip

Consider $\Psi\in \C_0(\O)^+\cap L^1$, such that $0<\Psi(x)$ for all $x\in \O$. 
Without loss of generality we can assume that $\|\Psi\|_\infty\le 1$.
With the aid 
of this function, we construct a new potential operator defined on $\C_b(\O)$ as
$$
V_\Psi(F)(x)=\GbO(\Psi F)(x).
$$
We point out that $\Psi F \in \C_0(\O)\cap L^1$. Then $V_\psi$ is a well defined, positive bounded linear
operator $V_\psi: \C_b(\O) \to \C_0(\O)\subset \C_b(\O)$, with norm
$$
\|V_\Psi\|_{\C_b(\O)}\le \inf\limits_{a>0} \left\{\int_{B(0,a)} g^{\beta}(0,y)\, dy+\left(\frac{C(d)}{a^{d-2}}\right)^\beta
\|\Psi\|_1\right\}.
$$
We also have that $V_\Psi$ satisfies the CMP in $\C_b(\O)$, because
$$
\int \left(V_\Psi(F)(x)-1\right)^+ \Psi(x) F(x)\, dx= \int \left(\GbO(\Psi F)(x)-1\right)^+ \Psi(x) F(x)\, dx\ge 0.
$$
Thus, if $V_\Psi(F)(x)\le 1$ on the set $\{z\in \O:\, F(z)\ge 0\}$, we conclude that $V_\Psi(F)(x)\le 1$ on the set
$\{z\in \O:\, \Psi(z) F(z)<0\}$, which is exactly the set where $F$ is negative. Thus, $V_\Psi$ satisfies the CMP.

Then, there exists a unique contraction resolvent 
$\VS_\Psi=(V_\Psi^\lambda)_\lambda$ of positive continuous linear operators on $\C_b(\O)$, 
such that $V_\Psi=V_\Psi^0$. 
Again, we can prove that this resolvent is continuous on $\lambda$ (see Proposition \ref{pro:2.5}). In particular
$$
V_\Psi=\lim\limits_{\lambda\downarrow 0} V_\Psi^\lambda.
$$
On the other hand, following
the proof of Lemma \ref{lem:3}, we have the density of $\mathcal{M}\cap \C_0(\O)$ in $\C_0(\O)$.
Then, there exists a Ray semigroup $\PS^\Psi=(P^\Psi_t)_{t\ge 0}$ and the associated Ray process $X^\Psi$, such 
that for all $F\in \C_b(\O)$ and all $\lambda\ge 0$
$$
V^\lambda_\Psi (F)(x)=\int e^{-\lambda t} P^\Psi_t(F)(x) \, dt=
\EE_x\left(\int e^{-\lambda t} F(X^\Psi_t)\, dt\right).
$$
The important case is $\lambda=0$, which gives for all $F\in \C_b(\O), x \in \O$
$$
\GbO(\Psi F)(x)=\int P^\Psi_t(F)(x) \, dt= \EE_x\left(\int F(X^\Psi_t)\, dt\right).
$$
Recall that $X^\Psi$ has c\`adl\`ag paths on $\hat \O$, the one point compactification of $\O$, and the
semigroup can be assumed to be extended to $\BB_b(\hat \O)$.

Finally, $X^\Psi$ is a Feller process as soon as we prove that the set of branching points is empty.
This is done exactly in the same way as we did it in Lemma \ref{lem:2.8}. So we summarize this in
the next proposition.

\begin{proposition} For every $\Psi\in \C^+_0(\O)\cap L^1$, which we assume is strictly positive, 
there exists a unique Feller process
$X^\Psi$, with c\`adl\`ag paths, taking values in $\hat\O$ such that its $0$-potential is $V_\Psi$, that is,
for all $F\in \C_b(\O)$ we have
$$
\GbO(\Psi F)(x)=V_\Psi(F)(x)=\EE_x\left(\int F(X^\Psi_t) \, dt\right).
$$
\end{proposition}

Next, we study the dependence on $\Psi$ for the resolvent, semigroup and process. This is done throughout
a time change, we explain it now. Assume that $\Psi_1\le \Psi_2$, both functions satisfying the above
requirements. We denote by $X^{\Psi_i}$ for $i=1,2$, the associated processes.

Consider the increasing process $\A_t=\int_0^t \frac{\Psi_1}{\Psi_2}\left(X^{\Psi_2}_s\right) \, ds$, where we 
assume that $0/0=0$. Since $\Psi_1\le \Psi_2$, 
then $\A$ is an increasing continuous process. If $\zeta$ is the hitting time of $\delta$ for $X^{\Psi_2}$, then
$\A$ is strictly increasing on $[0,\zeta]$ and it is constant on $[\zeta,\infty]$. 
We define $(\tau_t)_t$ the right continuous inverse of $\A$, which is continuous and strictly increasing
on $[0,\A_\zeta)$ and $\tau_t=\infty$ for $t\ge \A_\zeta$. Given that $\A_t\le t$, we have $\tau_t\ge t$. 

Consider $\tilde X$ the process obtained by time change from $X^{\Psi_2}$, that is $\tilde X_t=X^{\Psi_2}_{\tau_t}$.
$\tilde X$ is a Feller process, with c\`adl\`ag paths on $\hat \O$. The potential associated to $\tilde X$
is for $F\in \C_b(\O)^+$ (we extend $F$ to $\hat \O$ by $F(\delta)=0$)
$$
\begin{array}{ll}
\tilde V(F)(x)\hspace{-0.3cm}&=\EE_x\left( \int F(\tilde X_t) \, dt\right)=\EE_x\left( \int F(X^{\Psi_2}_{\tau_t}) \, dt\right)
=\EE_x\left( \int F(X^{\Psi_2}_{t}) \, d\A_t\right)\\
\hspace{-0.3cm}&=\EE_x\left( \int F(X^{\Psi_2}_{t}) \frac{\Psi_1}{\Psi_2}(X^{\Psi_2}_t) dt\right)=
V_{\Psi_2}(F \Psi_1/\Psi_2)=\GbO(\Psi_1 F)\\
\hspace{-0.3cm}&=V_{\Psi_1}(F).
\end{array}
$$
Therefore, $\tilde V=V_{\Psi_1}$ and from the uniqueness of the resolvent associated to this potential, 
we obtain that $\tilde V^\lambda=V_{\Psi_1}^\lambda$, for all $\lambda$. From here, we get that the Laplace
transform of the semigroups associated to $X^{\Psi_1}$ and $\tilde X$ coincide, which implies that
both processes have the same distribution: $X^{\Psi_1}\stackrel{\mathcal{L}}{=}\tilde X$.

\smallskip

We proceed now to construct the Feller process associated to $\GbO$. For that, take any function
$\Psi$ as above. The time change we propose is $\A_t=\int_0^t \left(\Psi(X^{\Psi}_s)\right)^{-1} \, ds$.
We point out that $\A$ is an increasing process, continuous in the interval $[0,\zeta)$, where
$\zeta>0$ is the hitting time of $\delta$ for the process $X^{\Psi}$. We put $\A_t=\A_{\zeta-}$ for 
$t\ge \zeta$, whenever $\zeta<\infty$. Again, we consider $(\tau_t)_t$ the right continuous
inverse of $\A$. Then, $(\tau_t)$ is strictly increasing and continuous on $[0,\A_{\zeta-})$. 

As before, we consider the Feller process $\XX=(X^\Psi_{\tau_t})_t$, taking values on
$\hat\O$ and lifetime $\zeta^{{}^{{}_\clubsuit}}=\A_{\zeta-}$. In principle this process
$\XX\,$ depends on $\Psi$.
As in the previous computation, we get
for every $F\in \CK(\O)$
\begin{eqnarray}
&V^{{}^{{}_\clubsuit}}\!(F)(x)&=\EE_x\left( \int F(\XXt) \, dt\right)=\EE_x\left( \int F(X^\Psi_{\tau_t}) \, dt\right)\nonumber\\
&&=\EE_x\left( \int F(X^\Psi_{t}) \, d\A_t\right)=\EE_x\left( \int F(X^\Psi_{t}) \frac{1}{\Psi}(X^\Psi_t) dt\right)
\label{eq:3.1}\\
&&=V_{\Psi}(F /\Psi)=\GbO(F).\nonumber
\end{eqnarray}
Existence in Theorem \ref{the:1} is shown, with the Feller process $\XX$.

We remark that the law of $\XX$, constructed above does not depend on the choice of $\Psi$. Indeed,
if we have two functions $\Psi_1, \Psi_2$, we consider $\Psi_3=\Psi_1\vee \Psi_2$ and we proceed
as above. Let us call $Z=X^{\Psi_3}, X=X^{\Psi_1}, Y=X^{\Psi_2}$. As above, $X$ is a time change of $Z$. 
More precisely,
$\tilde X=(Z_{\eta_t})_t$ has the same law as $X$, where $\eta$ is the inverse
of the increasing process $d\mathcal{B}_t=\frac{\Psi_1}{\Psi_3}(Z_t) dt$.

On the other hand, we denote by $\tilde{X}^{{}^{{}_\clubsuit}}\!\!,$ 
obtained from $\tilde X$ with a time change $(\tau_t)_t$, which
is the inverse of $d\A_t=\frac{1}{\Psi_1}(\tilde X_t) dt$. Notice that $\tilde{X}^{{}^{{}_\clubsuit}}\!\!=F(\tilde X)$
for some fixed measurable transformation $F$, and ${X}^{{}^{{}_\clubsuit}}\!=F(X)$. Therefore $\tilde{X}^{{}^{{}_\clubsuit}}$ 
and $\XX\,$ have the same law.
We can see $\tilde{X}^{{}^{{}_\clubsuit}}\!\!$ as a time change of $Z$. 
This time change is just the composition of the two
time changes, which is the inverse of $\mathcal{C}=(\A_{\mathcal{B}_t})_t$ and
$$
d\mathcal{C}_t=(d\A)_{\mathcal{B}_t}\, d{\mathcal{B}_t}=
\frac{1}{\Psi_1}(\tilde X_{\mathcal{B}_t})\, \frac{\Psi_1}{\Psi_3}(Z_t) dt=
\frac{1}{\Psi_1}(Z_{\eta_{{}_{\mathcal{B}_t}}})\, \frac{\Psi_1}{\Psi_3}(Z_t) dt=
\frac{1}{\Psi_3}(Z_t) dt,
$$
which is the time change from $Z$ to ${Z^{{}^{{}_\clubsuit}}}$. This shows that the law of $\XX$
and ${Z^{{}^{{}_\clubsuit}}}$ are the same. Analogously, the law of $Y^{{}^{{}_\clubsuit}}$ and 
${Z^{{}^{{}_\clubsuit}}}$ coincide, proving the claim.

With the same ideas we can prove uniqueness in Theorem \ref{the:1}. Assume that
$Y$ is a Feller process, with c\`adl\`ag paths on $\hat \O$, such that, for all $F\in \CK(\O)$
it holds
$$
\GbO(F)(x)=\EE_x\left( \int F(Y_t) \, dt\right).
$$
Then, using a time change with $(\tau_t)_t$, the inverse of $d\A_t=\Psi(Y_t) \, dt$, we see that
$X=(Y_{\tau_t})_t$ is a Feller process whose $0$-potential is, for $F\in \CK(\O)$,
$$
U(F)(x)=\GbO(\Psi F)(x)=V_\Psi(F)(x),
$$
and therefore, $V_\Psi$ is an extension of $U$ to $\C_b(\O)$. This shows that
$X$ has the same law as $X^\Psi$. 
Finally, $Y$ has the same law as $\XX$ and uniqueness is shown.

\section{Dimension $d=2$, proof of Theorem \ref{the:2}}
\label{sec:4}

In this section we shall prove Theorem \ref{the:2}. We first consider $\O$ a bounded regular open set in $\RR^2$. 
The main ideas are already exposed in previous sections. We need some bounds for the
Green potential on bounded domains of $\RR^2$. For that purpose we consider first the case of a 
ball. Consider $\O=B(0,1)$, then it is well known (see for example \cite{FukushimaOshimaTakeda2011}) that
$$
g_\O(x,y)=-\frac{1}{\pi} \left[ \log(\|x-y\|)-\log(\|x\|\|y-x_1^*\|)\right],
$$
where $x_1^*=\frac{x}{\|x\|^2}$ is the point dual of $x$ with respecto to $\partial B(0,1)$. By a scaling argument we
obtain the Green kernel for $B(0,R)$ as 
$$
g_{R}(x,y)=g_\O\left(\frac{x}{R},\frac{y}{R}\right)=-\frac{1}{\pi} \left[ \log(\|x-y\|)-\log(\|x\|\|y-x_R^*\|)+\log(R)\right],
$$
where now $x_R^*=R^2 \frac{x}{\|x\|^2}$.

Now, for the random walk we consider 
$$
a(x)=\sum\limits_{n=0}^\infty p_n(0,0)-p_n(0,x),
$$ 
which gives, roughly, the difference between the expected number of visits to $0$ minus the expected number
of visits to $x$, for the simple random walk in $\ZZ^2$. Notice that $a(0)=0$. 
The estimate we need on $a(x)$ is the following (see \cite{Lawler2010}, Theorem 4.4.4), for all $x\in\ZZ^2$
$$
\left |a(x)-\left(c_2\log(\|x\|\vee 1)+\frac{2\gamma+\log(8)}{\pi}\right)\right|\le \Psi(\|x\|),
$$
where $c_2=\frac{2}{\pi}$, $\gamma$ is the Euler constant and $\Psi:\RR^+\to \RR^+$ is a bounded
decreasing function, such that $\Psi(r)r^2$ is also bounded.

Now, the Green kernel for the simple random walk of a finite set $E$ is
$$
0\le \g_E(x,y)=\EE_x\left(\sum\limits_{k=0}^{R_E-1} \ind_y(\Ss_k)\right)=\EE_x(a(\Ss_{R_E}-y))-a(x-y).
$$
Consider as before, $\Zz=\sqrt{\frac{2}{n}}\,\ZZ^2$. Take $A_n=\Zz\cap B(0,1)$ 
and $E_n=\sqrt{n/2}A_n \subset B(0,\sqrt{n/2})$. So for $u\neq v \in A_n$ we have
$$
\begin{array}{l}
\g_{E_n}(u\sqrt{n/2},v\sqrt{n/2})=\EE_{u\sqrt{n/2}}(a(\Ss_{R_{E_n}}-v\sqrt{n/2}))-a(\sqrt{n/2}(u-v))\\
=\sum\limits_{\zeta\in \partial A_n} a(\sqrt{n/2}(\zeta-v)) 
\PP_{x\sqrt{n/2}}\left(\Ss_{R_{E_n}}=\sqrt{n/2} \, \zeta\right)-a(\sqrt{n/2}(u-v))\\
=\sum\limits_{\zeta\in \partial A_n} (a(\sqrt{n/2}(\zeta-v))-a(\sqrt{n/2}(u-v)))
\PP_{u\sqrt{n/2}}\left(\Ss_{R_{E_n}}=\sqrt{n/2} \, \zeta\right)\\
=\sum\limits_{\zeta\in \partial A_n} c_2(\log(\|\zeta-v\|)-\log(\|u-v\|))
\PP_{u\sqrt{n/2}}\left(\Ss_{R_{E_n}}=\sqrt{n/2}\, \zeta\right)+\R(u,v,n),
\end{array}
$$
where $\sup\limits_n\sup\limits_{w,z\in A_n}|\R(w,z,n)|<2\Psi(0)$ and for every $\e>0$ we have
{\small
$$
D(\e)=\sup\limits_{n\ge 1} n^2 \sup\{|\R(w,z,n)|:\, w,z \in A_n, \|w-z\|>\e, \|w\|\le 1-\e,\|z\|\le 1-\e\}<\infty
$$
}Now, for any $x,y \in B(0,1)$, we take  $u_n, v_n \in A_n$ any pair of sequences such that
$u_n\to x, v_n\to y$. Then, if $x\neq y, \|x\|<1, \|y\|<1$, we have as $n\to \infty$,
{\small
\begin{eqnarray}
\label{eq:4.0}
&\frac12 \g_{E_n}(u_n\sqrt{n/2},v_n\sqrt{n/2})\to\frac1\pi\int \log\left(\frac{\|\zeta-y\|}{\|x-y\|}\right)
\PP_x(W_{T_\O}\in d\zeta)=g_\O(x,y),
\end{eqnarray}
}
where $W$ is a BM. For the last equality see for instance \cite{FukushimaOshimaTakeda2011}, Example 1.5.1.

\medskip

In what follows we need some extra domination, to use the Dominated Convergence Theorem. 
For that purpose we consider $C_n=\Zz\cap B(0,2)$
and $F_n=\sqrt{n/2} C_n$. Take $n_0$ large enough, such that $d(A_n,\partial C_n)\ge 1/2$ for all $n\ge n_0$
($n_0=1$ actually works). 
If we take $n\ge n_0$ and $u,v \in A_n$, using the previous computations we have
{\small
\begin{equation}
\label{eq:4.1}
\frac12\g_{E_n}(u\sqrt{n/2},v\sqrt{n/2})\le\frac12 \g_{F_n}(u\sqrt{n/2},v\sqrt{n/2})\le \frac1\pi \Big|\log(\|u-v\|)\Big|+C,
\end{equation}
with $C= \frac1\pi \log(3)+ \Psi(0)$. 

Now, we define the Green potential associated to scaled random walk. We assume that $\beta\ge 1$
and $\alpha \in (0,2\pi)$

\begin{eqnarray}
&&\G^{n,(\beta)}_{A_n} (F)(x)=\sum\limits_{w\in A_n} \frac{2}{n}\, F(w+x-x(n)) 
\left[\frac{1}{2}\g_{E_n}\left(x(n)\sqrt{n/2},w\sqrt{n/2}\right)\right]^\beta\\
&&\G^{n,(\exp,\alpha)}_{A_n} (F)(x)=\sum\limits_{w\in A_n}  \frac{2}{n} \, F(w+x-x(n)) 
\, e^{\frac{\alpha}{2}\g_{E_n}\left(x(n)\sqrt{n/2},w\sqrt{n/2}\right)}\,
\end{eqnarray}
Using the convergence \eqref{eq:4.0} and domination \eqref{eq:4.1} we prove the following lemma.

\begin{lemma} Consider $\O=B(0,1)$. For $\beta\ge 1$, $\alpha\in (0,2\pi)$, $F\in \C_b(\O)$ and
$x\in \O$, we have
\begin{enumerate} 
\item $\lim\limits_{n\to\infty} \G^{n,(\beta)}_{A_n} (F)(x)=\GbO(F)(x)=\int_\O F(y) \gbO(x,y)\, dy$.
\item $0\le \lim\limits_{n\to\infty} \int \left(\G^{n,(\beta)}_{A_n} (F)(x)-1\right)^+ F(x) \, dx=
\int_\O \left(G^{(\beta)}_\O (F)(x)-1\right)^+ F(x) \, dx$.
\item $\lim\limits_{n\to\infty} \G^{n,(\exp,\alpha)}_{A_n} (F)(x)=G^{\,\exp,\alpha}(F)(x)=\int_\O F(y) e^{g_\O(x,y)}\, dy$.
\item $0\le \lim\limits_{n\to\infty} \int \left(\G^{n,(\exp,\alpha)}_{A_n} (F)(x)-1\right)^+ F(x) \, dx=
\int_\O \left(G^{\,\exp,\alpha} (F)(x)-1\right)^+ F(x) \, dx$.
\end{enumerate}
\end{lemma}
The immediate consequence of this lemma is that
\begin{proposition} For $\O=B(0,1),\beta\ge 1, \alpha\in (0,2\pi)$ the bounded positive linear operators
$G^{(\beta)}_\O, G^{\,\exp,\alpha}_\O$
satisfy the CMP in $\C_b(\O)$.
\end{proposition}
The rest of the proof of Theorem \ref{the:2}, for $\O=B(0,1)$, is obtained by following the lines in Section \ref{sec:2},
showing the analogous of Lemma \ref{lem:2.4}, Proposition \ref{pro:2.5}, Lemma \ref{lem:3}, 
Theorem \ref{the:2.7} and finally
Lemma \ref{lem:2.8}.

Finally, we sketch the proof of Theorem \ref{the:2} for a general bounded regular open set $\O\subset\RR^2$.
For that purpose, we consider a large ball $\widetilde{\O}=B(0,R)$, such that $\O\subset B(0,R/2)$. We take
$A_n=\Zz\cap \O, E_n=\sqrt{n/2}\, A_n$ and $C_n=\Zz\cap \widetilde \O, F_n=\sqrt{n/2}\, C_n$. The idea is to
use $g=g_{\widetilde \O}, \g=\g_{F_n}$, as we did with $g_{\RR^d}, \g_{\ZZ^d}$ in Section \ref{sec:2}.

The main relations we need are, for $u,v \in A_n, x,y \in \O$
$$
\begin{array}{l}
\g_{E_n}(u,v)=\g_{F_n}(u,v)-\EE_x(\g_{F_n}(\Ss_{R_{E_n}},v))\le \g_{F_n}(u,v)\\
\\
g_\O(x,y)=g_{\widetilde \O}(x,y)-\EE_x(\g_{\widetilde \O}(W_{T_{\O}},y))\le g_{\widetilde \O}(x,y).
\end{array}
$$
These representations provide the convergence and domination we need to finish the proof when $\O$
is a bounded regular open set. The general case, that is when $\O$ is an unbounded regular Greenian domain,
is done as in Section \ref{sec:3}.

\section{Existence of a density for the semigroup}
\label{sec:density}
In this section we show that the semigroup $\mathscr{P}=(P_t)_t$ whose $0$-potential is $\GbO$, has 
a density $p(t,x,y)$ with respect to Lebesgue measure, that is, for all $f\in \C_b(\O)$ it holds
$$
P_t(f)(x)=\int_\O f(y) \, p(t,x,y)\, dy.
$$
We restrict ourselves to the case $\O$ is a bounded regular open set of $\RR^d$ and $d\ge 2$. 

The arguments given below, with some minor modifications, will work as well for an unbounded 
regular open set $\O\subset \RR^d$, with $d\ge 3$ and for  
the semigroup associated to $G_\O^{\exp,\alpha}$, with $\O$ a regular bounded open set in dimension $d=2$. 

In what follows, we shall use some results for symmetric Markov process
given in \cite{FukushimaOshimaTakeda2011}.
We denote by $\U=(U^\lambda)_\lambda$ the associated resolvent. We know that $\mathscr{P}$ is a Feller
semigroup and $\U$ is a continuous resolvent (in $\lambda$). In particular, for every $f\in \C_b(\O)$
the function $(t,x)\to P_t(f)(x)$ is continuous.

\begin{proposition} There exists a function $u:\RR^+\times\O\times \O\to (0,\infty]$, such that
\begin{enumerate}
\item $u(\lambda,\bullet,\bullet)$ is a symmetric function, bounded above by $\gbO$.
\item $u$ is continuous on $[0,\infty)\times (\O\times \O \setminus \{(x,x):\, x \in \O\})$, and 
$u(0,x,y)=\gbO(x,y)$ for all $x,y$.
\item $u(\lambda,x,\bullet)$ is a density for $U^\lambda$, that is, for all $f\in \C_b(\O)$ it holds
$$
U^\lambda (f)(x)=\int_\O f(y) \, u(\lambda,x,y)\, dy.
$$
\item The semigroup $\mathscr{P}$ has an extension to $L^2(\O,dx)$, which is symmetric and continuous.

\item For every $t>0, x\in \O$ the measure $P(t,x,dy)$ has a density $p(t,x,y)$ 
with respect to Lebesgue measure. The semigroup $\mathscr{P}$ has a continuous extension to $L^p(\O,dx)$ 
for all $1\le p\le \infty$. Every $P_t$ is a contraction in $L^p(\O,dx)$.

\item For all $t, x\neq y$, the function $P_t(\gbO(\bullet,x))(y)$ is well defined and finite. This function
is decreasing and right continuous in $t$. This function is also measurable in the three variables. 
The function $\nu(t,x,y)=\gbO(y,x)-P_t(\gbO(\bullet,x))(y)$ is increasing
and right continuous on $t$. For all $\lambda\ge 0$ and $x, y$, we have 
\begin{equation}
\label{eq:5.0}
u(\lambda,x,y)=\int_0^\infty e^{-\lambda t}\, \nu(dt,x,y).
\end{equation}
In particular, 
$$
\gbO(x,y)=\int_0^\infty \nu(dt,x,y).
$$
Also, for all $f\in \C_b(\O)$ it holds
$$
U^\lambda (f)(x)=\int_\O \int  e^{-\lambda t} f(y) \nu(dt,x,y)\, dy.
$$
The measure $\nu(dt,x,y)$ is absolutely continuous in $t$, for $x$ and $dy$-a.e. Its density (with respect to $t$) is
$$
\frac{\partial}{\partial t} \nu(t,x,y)=p(t,x,y).
$$
In particular, for all $x$ and $dy$-a.e. it holds
$$
\gbO(x,y)=\int_0^\infty p(t,x,y) \, dt.
$$
\item For all $f\in \C_b(\O), x\in \O$ the function $t\to \int_\O P_t(\gbO(\bullet,x))(y) f(y) \, dy$ is $\C^1([0,\infty))$
and 
$$
-\frac{\partial}{\partial t} P_t(\GbO(f)(x)=-\frac{\partial}{\partial t} \int_\O P_t(\gbO(\bullet,x))(y) f(y) \, dy=P_t(f)(x).
$$
As a special case we obtain
$$
\lim\limits_{h\downarrow 0} \frac{P_h(\GbO(f))-\GbO(f)}{h}=-f,
$$
holds uniformly. This means that $\GbO(f)$ is in the domain of $\mathcal{L}$ the infinitesimal generator of $(P_t)_t$ and
$$
\mathcal{L}(\GbO(f))=-f.
$$
\end{enumerate}
\end{proposition}

\begin{proof} 
$(1)$--$(3)$. From the resolvent equation, we have for all $f\in \C^+_b(\O)$ and all $x$
$$
0\le U^\lambda(f)(x)\le \GbO(f)(x).
$$
By the Riesz representation theorem, there exists a measure $\rho(dy)$, which depends on $\lambda,x$
such that $U^\lambda(f)(x)=\int_\O f(y) \rho(dy)\le \int_\O f(y) \gbO(x,y) \, dy$. The conclusion is that $\rho$ is absolutely
continuous with respect to the Lebesgue measure, with a density $u(y)=u(\lambda,x,y)$, which is bounded
by $\gbO(x,y)$. Again the resolvent equation shows that, for $x\neq y$
\begin{equation}
\label{eq:4.1}
\gbO(x,y)-u(\lambda,x,y)=\lambda \int_\O u(\lambda,x,z) \gbO(z,y) \, dz=\lambda U^\lambda(\gbO(y,\bullet))(x),
\end{equation}
which also proves that $u(\lambda,x,\bullet)$ has a continuous version on $\O\setminus \{x\}$.
Here we notice that 
$$
U^\lambda(\gbO(y,\bullet))(x)\le \GbO(\gbO(y,\bullet))(x)=
\int_\O \gbO(y,z)\gbO(z,x)\, dz\le \int_\O g^\beta(y,z) g^\beta(z,x)\, dz,
$$ 
which is finite for $y\neq x$. This representation also proves the continuity of $u$ on 
$[0,\infty)\times (\O\times \O \setminus \{(x,x):\, x \in \O\})$.

The symmetry of $u(\lambda,\bullet,\bullet)$, follows from the symmetry of the operator $U^\lambda$
with respect to Lebesgue measure: for all $f,g \in \C_b(\O)$ it holds 
$$
\int_\O U^\lambda(f)(x) g(x) \, dx=\int_\O U^\lambda(g)(x) f(x) \, dx.
$$ 
This can be shown by using that $U^\lambda=(\II+\lambda \GbO)^{-1}\GbO$. Indeed, for 
small $\lambda<\|\GbO\|_{\C_b(\O)}$, we can use the expansion
$$
U^\lambda=\sum\limits_{n=0}^\infty (-1)^{n} \lambda^n (\GbO)^{n+1},
$$
which proves that $U^\lambda$ satisfies the desired symmetry for small $\lambda$, given that
$\GbO$ has a symmetric kernel $\gbO$. The resolvent equation allows
us to extend this property to all $\lambda$.

\noindent $(4)$ We first show that for all $f,g\in \C_b(\O)$ and all $t\ge 0$, we have the desired
symmetry
$$
\begin{array}{l}
\int_\O U^\lambda(f)(x) g(x) \, dx=\int_0^\infty e^{-\lambda t} \int_\O P_t(f)(x) g(x)\, dx
=\int_\O U^\lambda(g)(x) f(x) \, dx\\
\\
=\int_0^\infty e^{-\lambda t} \int_\O P_t(g)(x) f(x)\, dx.
\end{array}
$$
The uniqueness of Laplace transform implies that 
$$
\int_\O P_t(f)(x) g(x)\, dx=\int_\O P_t(g)(x) f(x)\, dx
$$
holds for almost all $t$. Since the functions $t\to \int_\O P_t(f)(x) g(x)\, dx$,  
$t\to \int_\O P_t(g)(x) f(x)\, dx$ are continuous, we conclude they are equal.

Now, we show that $\mathscr{P}$ has an extension to $L^2$. Using the Cauchy-Schwarz
inequality, we show that for all $f\in\C_b(\O)$
$$
\begin{array}{l}
\int_\O (P_t(f)(x))^2 \, dx\le \int_\O (P_t(\ind)(x))P_t(f^2)(x) \, dx\le \int_\O P_t(f^2)(x) \, dx\\
\\
=\int_\O P_t(\ind)(x) f^2(x) \, dx\le \int_\O f^2(x) \, dx.
\end{array}
$$
This shows that $\mathscr{P}$ has a continuous extension to $L^2(\O,dx)$.

\noindent $(5)$ The existence of densities is a straightforward 
consequence of the symmetry of $\mathscr{P}$ in $L^2(\O,dx)$, and 
Theorem 4.1.2 and Theorem 4.2.4 in \cite{FukushimaOshimaTakeda2011}.

Now, consider $p=1$ and $f \in \C_b(\O)$, we have again 
$$
\int_\O \left|P_t(f)(x)\right|\, dx\le \int_\O P_t(|f|)(x) \, dx=\int_\O P_t(\ind)(x) |f(x)|\, dx\le\|f\|_1.
$$
This shows that $P_t(f)\in L^1(\O,dx)$ and $\|P_t\|_1\le 1$. 
The case $p=\infty$ is obvious since for all $f\in L^\infty(\O,dx)$ we have 
$$
|P_t(f)(x)|\le \int_\O p(t,x,y) |f(y)| \, dy \le \|f\|_{L^\infty} P_t(\ind)(x)\le \|f\|_{L^\infty}.
$$
Moreover, in this case we have 
$$
 \|P_t(f)\|_\infty \le \|f\|_{L^\infty}.
$$
For a general $1<p<\infty$ and $f\in \C_b(\O)$,
we have from H\"older's inequality with $q$ the conjugated of $p$
$$
\begin{array}{l}
\int_\O |P_t(f)(x)|^p\, dx\le \int_\O (P_t(|f|)(x))^p\, dx \le  \int_\O (P_t(|f|^p)(x)) (P_t(\ind)(x))^{p/q} \, dx\\
\\
\le \int_\O (P_t(|f|^p)(x)) \, dx=\int_\O (P_t(\ind)(x)) |f|^p(x) \, dx\le \int_\O |f|^p(x) \, dx.
\end{array}
$$
This proves the claim. 

\noindent $(6)$ Consider $f_k\in \C^+_0(\O)$ whose support is contained in $B(y,1/k)$, and $\int f_k(x) dx=1$, that
is, $f_k$ is an approximation of $\delta_y$. The function $\GbO(f_k)(x)$ is a nonnegative supermedian continuous 
function and also it is $\GbO(f_k)(x) \wedge N$. Then, it is known that $(t,z)\to P_t(\GbO(f_k) \wedge N)(z)$
is a jointly continuous function of $(t,z)$ and it is decreasing in $t$. 
We notice that $\GbO(f_k) \wedge N$ converges in $k$ to the continuous and bounded function 
$\gbO(y,\bullet) \wedge N$ and therefore $(t,z)\to P_t(\gbO(y,\bullet) \wedge N)(z)$ is also
jointly continuous and decreasing in $t$ (continuity follows from the Feller property). 

The monotone convergence Theorem allows us to conclude that
$$
(t,z)\to P_t(\gbO(y,\bullet))(z)
$$
is jointly measurable, decreasing in $t$ and bounded by $\gbO(y,z)$, that is, for
all $t,x,z$
$$
P_t(\gbO(y,\bullet))(z)\le \gbO(y,z).
$$
The function $t \to P_t(\gbO(y,\bullet))(z)$ is decreasing and lower semicontinuous, which implies that
it is right continuous.

The same reasoning shows that 
$$
U^\lambda(\gbO(y,\bullet))(x)=\int_0^\infty e^{-\lambda t} P_t(\gbO(y,\bullet))(x) dt.
$$
Thus, from \eqref{eq:4.1}
$$
\gbO(x,y)-u(\lambda,x,y)=\lambda \int e^{-\lambda t} P_t(\gbO(y,\bullet))(x) dt,
$$
or equivalently
$$
u(\lambda,x,y)=\lambda \int e^{-\lambda t} (\gbO(x,y)-P_t(\gbO(y,\bullet))(x)) dt.
$$
The symmetry between $x$ and $y$ also shows that
\begin{equation}
\label{eq:5.2}
u(\lambda,x,y)=\lambda \int e^{-\lambda t} (\gbO(x,y)-P_t(\gbO(x,\bullet))(y)) dt.
\end{equation}
Integrating by parts the RHS we obtain
$$
u(\lambda,x,y)=\int e^{-\lambda t} \nu(dt,x,y),
$$
which proves \eqref{eq:5.0}. Notice that $\nu(dt,x,y)$ is also symmetric in $x,y$.

We can take the limit as $\lambda \downarrow 0$ in \eqref{eq:4.1}, to show that for $x\neq y$
$$
\gbO(x,y)=\lim\limits_{\lambda \downarrow 0} u(\lambda,x,y)=\int \nu(dt,x,y)=\nu(\RR^+,x,y)
$$
which gives the desired result. That is, for $x\neq y$ the measure $\nu(dt,x,y)$ is a finite measure.
Also we conclude that for $x\neq y$ the limit
$$
\lim\limits_{t\to \infty} P_t(\gbO(y,\bullet))(x)=0.
$$
On the other hand, for any $f\in \C_b(\O)^+$ we have
$$
U^\lambda(f)(x)=\int_0^\infty e^{-\lambda t} P_t(f)(x) \, dt=\int_\O \int_0^\infty e^{-\lambda t} p(t,x,y) \, dt \, f(y)\, dy.
$$
This implies that for all $x, \lambda\ge 0$ 
and for $dy$-a.e. it holds 
$$
u(\lambda, x, y)=\int_0^\infty e^{-\lambda t} p(t,x,y) \, dt.
$$
Integrating by parts the RHS we get
$$
u(\lambda, x, y)=\lambda\int_0^\infty e^{-\lambda t} \int_0^t p(s,x,y)\, ds \, dt.
$$
Comparing with \eqref{eq:5.2}, we get that $dt$-a.s. it holds
$$
\gbO(x,y)-P_t(\gbO(x,\bullet))(y)=\int_0^t p(s,x,y)\, ds.
$$
Since the left side is monotone in $t$ and the right side is continuous in $t$, we conclude they are
equal for all $t$. So, in this situation the measure $\nu(dt,x,y)$ is absolutely continuous and its derivative
is $p(t,x,y)$ (of course $dt$-a.s.). In particular, for all $x, \lambda$ and $dy$-a.e., we have
$$
\begin{array}{l}
\gbO(x,y)=\int_0^\infty p(t,x,y) \, dt,\\
\\
u(\lambda,x,y)=\int_0^\infty e^{-\lambda t} p(t,x,y) \, dt.
\end{array}
$$

\noindent $(7)$. Notice that 
$$
\begin{array}{ll}
\int P_t(\gbO(\bullet,x))(y) f(y) dy&\hspace{-0.3cm}=\int P_t(\gbO(\bullet,y))(x) f(y) dy=P_t(\GbO(f))(x)\\
\\
&\hspace{-0.3cm}=\int_t^\infty P_s(f)(x) \, ds.
\end{array}
$$
The function $t\to P_t(f)(x)$ is continuous, so $P_t(\GbO(f))(x)$ is $\C^1([0,\infty))$.

To get the uniform convergence we have
$$
P_h(G(f))(y)-P_0(G(f))(y)=-\int_0^h P_s(f)(y) dy.
$$
Since $P_s(f)$ converges uniformly to $f$, as $s\downarrow 0$, we get the result.
\end{proof}

\begin{appendix}
\section{Some Matrix results.}
\label{app:1}

In this appendix, we discuss the results for potential matrices associated to finite state 
Markov Chains (or continuous time MC) used in this article.

\begin{definition} A nonnegative matrix $U$, indexed by finite set $E$, is called a potential if it 
satisfies the Complete Maximum Principle (CMP): 
for all $v\in \RR^E$ if it is verified that $(Uv)_j\le 1$ on the coordinates $j$ where $v_j\ge 0$, then
it holds that $(Uv)_j\le 1$ at all coordinates $j$.
\end{definition}

A nonnegative nonsingular matrix is a potential if and only if its inverse 
is a row diagonally dominant $M$-matrix, that is,
\begin{enumerate}
\item for all $i\neq j \in E$, we have $(U^{-1})_{ij}\le 0$ and
\item for all $i\in E$ the row sum $\sum_j (U^{-1})_{ij}\ge 0$.
\end{enumerate}
This is part of the fundamental  Choquet-Deny  paper see \cite{choquet1956} (see also Theorem 2.9 in \cite{libroDMSM2014}).

If $U$ is a nonsingular potential then $Q=-U^{-1}$ is the generator of a continuous time
transient Markov chain $X=(X_t:\, t\ge 0)$ with state space $E$, such that
$$
U_{ij}=\EE_i\left(\int_0^\infty \ind_{j}(X_t) \, dt\right),
$$
that is, $U$ is the Green potential of $X$. This property characterizes nonsingular potential matrices 
(see for example Theorem 2.27 in \cite{libroDMSM2014}). 

On the other hand, it is straightforward to show that $U^{-1}=k(\II-P)$ for some constant $k$ and a 
sub-stochastic kernel $P$, and the 
spectral radius of $P$ is strictly bounded by one. This last property is equivalent to the convergence of the series
$\sum_{n\in \NN} P^n$ and moreover $U=\frac1k \sum_{n\in \NN} P^n$. 
Notice that since $U=\frac1k \sum_{n\in \NN} P^n$, we get that $P$ is irreducible if and only if $U>0$.
On the other hand, if we can choose $k=1$, then $U$ is the potential
of a Markov chain, whose transition kernel is exactly $P$.

The following result is a characterization of potential matrices, which is crucial for our work. 
Here we denote by $\langle\, ,\, \rangle$ the euclidian inner product in $\RR^E$.

\begin{proposition} 
\label{pro:3}
Assume that $~U$ is an entrywise nonnegative matrix indexed by a finite set $E$. If
\begin{equation}
\label{eq:3}
\forall v\in \RR^E:\,\, \langle (Uv-1)^+, v \rangle=\sum\limits_{j\in E} \left(\sum\limits_{k\in E} U_{jk} v_k-1\right)^+ v_j\ge 0,
\end{equation}
then $U$ is a potential. 

Conversely, if $U, U^t$ are potential matrices (for example if $U$ is a symmetric potential) then (\ref{eq:3}) holds.
\end{proposition}

\begin{proof} We prove that (\ref{eq:3}) is sufficient for $U$ to be a potential. 
Take a vector $v\in \RR^E$ that satisfies $\forall j,\, v_j\ge 0 \Rightarrow (Uv)_j\le 1$.
Then, we have 
$$
0\le \langle (Uv-1)^+, v \rangle= \langle (Uv-1)^+, -v^- \rangle\le 0.
$$
This implies that if $v_j<0$ then $((Uv)_j-1)^+=0$ and therefore $(Uv)_j\le 1$ proving that $U$ satisfies the 
CMP.

Conversely, assume that $U, U^t$ are potential matrices. Consider $a>0$, then the matrix $U(a)=a\II+U$
is nonsingular and satisfies the CMP. So, $M(a)=(U(a))^{-1}=k(a)(\II-P(a))$, for some 
constant $k(a)$ and a double substochastic matrix $P(a)$ (here we have used that $U, U^t$ are potentials). 
We define $\mu(a)=M(a)\ind\ge 0$ and  $\xi(a)=U(a)(v-\mu(a))=U(a)v-1$ to get
$$
\begin{array}{l}
\langle (U(a)v-1)^+, v \rangle=\langle (U(a)v-U(a)\mu)^+, v \rangle=\langle \xi^+(a), M(a)\xi(a)+\mu(a) \rangle\\
=\langle \xi^+(a),k(a) \xi(a) +\mu(a) \rangle-k(a) \langle \xi^+(a),P(a)\xi(a) \rangle\\
=k(a)\left(\langle \xi^+(a), \xi^+(a) \rangle-\langle \xi^+(a),P(a)\xi(a) \rangle\right)+\langle \xi^+(a),\mu(a) \rangle.
\end{array}
$$
Since $P(a)\ge 0$, we get 
$$
\begin{array}{ll}
\langle \xi^+(a),P(a)\xi(a) \rangle\hspace{-0.3cm}&\le \langle \xi^+(a),P(a)\xi^+(a) \rangle= 
\langle \xi^+(a),\frac12(P(a)+P^t(a))(\xi)^+(a) \rangle\\
\hspace{-0.3cm}&\le \langle \xi^+(a), \xi^+(a) \rangle.
\end{array}
$$
The last inequality holds because the nonnegative symmetric matrix $\frac12(P(a)+P^t(a))$ 
is sub-stochastic and therefore its spectral radius is smaller than 1, which
implies that for all $z\in \RR^d$ it holds $ \langle z,\frac12(P(a)+P^t(a))z \rangle\le \langle z, z \rangle$. We get
$$
\langle (U(a)v-1)^+, v \rangle\ge \langle (U(a)v-1)^+,\mu(a) \rangle\ge 0.
$$
The result follows by taking $a\downarrow 0$.
\end{proof}
The next result is the analogous of Theorem \ref{the:1} for matrices 
and it can be found in Theorem 6.5 in \cite{libroDMSM2014} (see also \cite{dell2009}). 
Recall that given a matrix $A$ and a real function $F$, the $F$-Hadamard 
function  of $A$ is defined entrywise as $(F(A))_{ij}=F(A_{ij})$. When $F(x)=x^\beta$ is
a power function, we denote $A^{(\beta)}=F(A)$.

\begin{proposition} 
\label{pro:4}
Assume that $~U$ is a (nonsingular) potential matrix. Then, 
\begin{enumerate}
\item if $\beta\ge 1$, the Hadamard power $U^{(\beta)}$ is also a (nonsingular) potential matrix.
\item For all $\alpha>0$ the Hadamard exponential $\exp(\alpha U)$ is also a (nonsingular) potential matrix.
\end{enumerate}
\end{proposition}
\end{appendix}

\section*{Acknowledgement} The authors are thankful to Martin Barlow, who pointed out the results in 
\cite{FukushimaOshimaTakeda2011} for the existence of density in the context of symmetric Markov Processes.
S.M., J.S.M. and P.V. were supported in part by CONICYT BASAL  AFB170001. M.D. was supported 
in part by Milenio NC120062.

\end{document}